\documentclass[12pt]{article}
    \usepackage{url}
    \usepackage{verbatim}
    \usepackage[titletoc]{appendix}
    \usepackage{graphicx}
	\usepackage{amsmath}
	\usepackage{amsthm}
	\usepackage{amsfonts}
    \usepackage{nicefrac}
    \usepackage{longtable}
    \usepackage{color}
    \usepackage{bm}
    \usepackage{booktabs}
    \usepackage{graphicx,subfigure, amssymb, subcaption,graphics}
	\usepackage[colorlinks, linkcolor=blue,
	anchorcolor=red,
	citecolor=green
	]{hyperref}
	\usepackage{comment}
	\usepackage{authblk}

    \def\B{\mbox{\boldmath $B$}}
    
    \def\D{\mbox{\boldmath $D$}}
    \def\E{\mbox{\boldmath $E$}}
    \def\e{\mbox{\boldmath $e$}}
    \def\f{\mbox{\boldmath $f$}}
    \def\g{\mbox{\boldmath $g$}}
    
    \def\J{\mbox{\boldmath $J$}}   
    \def\n{\mbox{\boldmath $n$}}
    \def\Q{\mbox{\boldmath $Q$}}        
    \def\u{\mbox{\boldmath $u$}}
    
    \def\w{\mbox{\boldmath $w$}}

    \def\v{\mbox{\boldmath $v$}}
    \def\rmd{{\rm d}}

    \def\L{{\mathbf L}}
    
    \def\H{{\mathbf H}}

     \def\0{\mbox{\boldmath $0$}}
	
    \newtheorem{thm}{Theorem}[section]

	\newtheorem{lem}[thm]{Lemma}

    \numberwithin{equation}{section}
    
    \allowdisplaybreaks
\begin{document}

\title{Optimal error estimates of a second-order temporally finite element method for electrohydrodynamic equations   }

\author{Shengfeng Wang \thanks{School of Mathematical Sciences, University of Electronic Science and Technology of China, Chengdu 611731, China, 202111110509@std.uestc.edu.cn}, \ \
	Maojun Li \thanks{School of Mathematical Sciences, University of Electronic Science and Technology of China, Chengdu 611731, China, limj@uestc.edu.cn},  \ \
	Zeyu Xia \thanks{School of Mathematical Sciences, University of Electronic Science and Technology of China, Chengdu 611731, China, zeyuxia@uestc.edu.cn,  Corresponding Author} }
	
\date{} 

\maketitle

\begin{abstract}

In this work, we mainly present the optimal convergence rates of the temporally second-order finite element scheme for solving the electrohydrodynamic equation. 
%In this work, we develop a fully discrete scheme for the electrohydrodynamic equation by using the second-order backward differentiation formula and finite element methods. For the proposed scheme, we establish its unique solvability, discrete charge conservation, unconditional energy stability, and particularly optimal convergence rates.
Suffering from the highly coupled nonlinearity, the convergence analysis of the numerical schemes for such a system is rather rare, not to mention the optimal error estimates for the high-order temporally scheme. 
To this end, we abandon the traditional error analysis method following the process of energy estimate, which may lead to the loss of accuracy. Instead, we note that the charge density also possesses the ``energy'' decaying property directly derived by its governing equation, although it does not appear in the energy stability analysis. This fact allows us to control the error terms of the charge density more conveniently, which finally leads to the optimal convergence rates.
Several numerical examples are provided to demonstrate the theoretical results, including the energy stability, mass conservation, and convergence rates.

\ \


\end{abstract}

\section{Introduction}

The hydrodynamic behavior in electromagnetic fields plays a crucial role in the engineering applications, and its research is quite popular in scientific computation in recent years. These phenomena can be governed by the incompressible Navier--Stokes (NS) equation
\begin{align}
	\rho_f(\partial_t\u + \u\cdot\nabla\u) - \eta\Delta\u + \nabla p &= \g + \f,	\label{NS-total1}	\\
	\nabla\cdot\u &= 0.	\label{NS-total2}
\end{align}
In this system, $(\u, p)$ represents the velocity-pressure pair; the positive constants $\rho_f$ and $\eta$ denote the density and viscosity of the fluid, respectively; $\g$ stands for the gravity; and $\f$ is the body force depending on different fields and fluids. 
The study of various field-fluid interactions can mainly be divided into three categories: One category is magnetohydrodynamics \cite{Bellan2008}, which describes the interactions between the magnetic fields and conductive fluid, where $\f$ is the Lorentz force; another category is ferrohydrodynamics \cite{Rosensweig2013}, in which the fluid can be magnetized in the presence of a magnetic field and the body force $\f$ is the polarization force; the last category, the subject of this work, is electrohydrodynamics, where the fluid is influenced by the electric field and is driven by the Coulomb force $\f$. 

As an important branch of the hydrodynamics, the electrohydrodynamics (EHD) has been widely applied in many fields. For example, in \cite{Li2025} the authors designed a self-powered triboelectric-EHD pump combining  a soft EHD pump driven by an electrostatic generator. In \cite{Feng2004} the authors utilize the EHD phenomena to enhance and control mass and heat transfer in both terrestrial and microgravity environments, which facilitates the flow distribution control through the conduction pumping. In addition, the study on EHD can serve the electrostatic precipitator \cite{Higuera1999}, laminar and turbulent mixing \cite{Castellanos1991, Perez2014} and more other applications \cite{Collins2008, Hayati1986}. 

To describe the EHD phenomena precisely, one must combine the NS equation \eqref{NS-total1}-\eqref{NS-total2} with the Maxwell's equations. It is well-known that the Coulomb force takes the form of $\f=\rho\E$, where $\rho$ represents the charge density and $\E$ denotes the electric fields. 
Thus, we now need to introduce the relationship between $\rho$ and $\E$. With the help of the constitutive relation $\nabla\cdot\D=\rho$ and $\E = \epsilon \D$, we naturally obtain 
\begin{align}\label{constitutive-1}
	\epsilon\nabla\cdot \E =\rho,
\end{align} 
where $\D$ is the electric displacement and $\epsilon$ is the dielectric permittivity assumed to be a positive constant in isotropic materials. Furthermore, if we assume the fluid to be non-magnetic, the Faraday's law $\nabla\times\E = -\partial_t \B= \0$ shows that the electric field $\E$ is curl-free, which allows one to define 
\begin{align}\label{constitutive-2}
	\E = -\nabla\phi,
\end{align}
with an artificial electric potential $\phi$. 
If the charge density is conserved, an application of Am\`{e}re's circuital law yields
\begin{align}\label{constitutive-3}
	\partial_t\rho + \nabla\cdot\J = 0,
\end{align} 
with current density $\J = \rho\u + \sigma\E - D\nabla\rho$, where $\sigma$ and $D$ are the drift-induced conductivity and charge diffusion coefficient, respectively. More details can be referred to \cite{Monk2003, Pan2020, Wang2024c}.

Consequently, combining NS equation \eqref{NS-total1}-\eqref{NS-total2} with the constitutive relations \eqref{constitutive-1}-\eqref{constitutive-3}, the governing equations of the EHD system can be formulated as follows:
\begin{align}
	-\epsilon\Delta\phi &= \rho,	\label{PDE1}	\\
	\partial_t\rho + \nabla\cdot(\rho\u) - D\Delta\rho + \frac{\sigma}{\epsilon}\rho & = 0, 	\label{PDE2}	\\
	\partial_t\u + \u\cdot\nabla\u - \eta\Delta\u + \nabla p &=-\rho\nabla\phi,	\label{PDE3}	\\
	\nabla\cdot\u &= 0,	\label{PDE4}
\end{align}
in $\Omega\times(0, T)$, where $\Omega$ is a bounded convex polyhedral domain in $\mathbb{R}^3$ (polygonal domain in $\mathbb{R}^2$) and $T>0$ represents the terminal time.
Without loss of the generality, we assume $\rho_f=1$, and note that for the general case $\rho_f>0$ there are no essential differences in the analysis. 
To close the system, we consider EHD equation \eqref{PDE1}-\eqref{PDE4} subject to the following initial and boundary conditions
\begin{align*}
	\rho|_{t=0}=\rho_0, \qquad 	\u|_{t=0} = \u_0,
	\qquad (\n\cdot\nabla\phi)|_{\partial\Omega}=0,
	\qquad (\n\cdot\nabla\rho)|_{\partial\Omega}=0,
	\qquad \u|_{\partial\Omega}=\0,
\end{align*} 
where $\rho_0$ and $\u_0$ are given, and $\n$ is the unit outer normal vector on the boundary $\partial\Omega$.
Particularly, from \eqref{PDE1} we obtain the charge conservation property:
\begin{align*}
	\int_\Omega\rho \rmd x = -\epsilon\int_\Omega\Delta\phi \rmd x 
		= -\epsilon \int_{\partial\Omega} \n\cdot\nabla\phi \rmd s =0, 
\end{align*}
which is actually consistent with the assumption. Moreover, we assume $\int_\Omega\phi\rmd x=0$ to guarantee the unique solvability of \eqref{PDE1}.
The regularity analysis for such a system can be found in \cite{Cai2024, Deng2010, Li2009, Smit2016} and the references therein.

Concerning the numerical methods, there exist extensive works focusing on the energy stability and computational efficiency.
In \cite{Pan2023, Pan2020, Pan2021, Zhu2024}, the authors developed a class of efficient numerical algorithms for solving the EHD system with constant/variable density and 
conductivity. 
Moreover, to solve the variable density EHD system, some fully decoupled linear and energy-stable schemes based on the ``zero-energy-contribution'' method were proposed in \cite{He2023, Wang2024}.
Although the numerical investigation for the EHD system is numerous, the optimal error estimates are quite rare. In \cite{He2023b}, the authors provided a temporal error analysis for the first-order semi-discrete scheme in the two-dimensional case based on the scalar auxiliary variable approach. More recently, in \cite{Li2024} the authors designed temporally first- and second-order fully decoupled schemes, and proved the first-order accuracy in time and the sub-optimal convergence rates in space.
From the references above, it can be seen that obtaining an optimal error estimate for the numerical scheme is highly desirable, especially for achieving the high-order temporally accuracy.

To this end, we propose a fully discrete scheme by utilizing the second-order backward differentiation formula (BDF) and finite element methods in the work, and then present a rigorous error estimate for the proposed scheme. Additionally, we prove the unconditional energy stability, unique solvability, and discrete charge conservation, which are the essential properties of the continuous system \eqref{PDE1}-\eqref{PDE4}. 
To achieve the optimal convergence rates, we do not adopt the traditional approach following energy stability, which may lead to losing accuracy. Instead, we discover that the charge density also satisfies the following ``energy'' decaying property:
\begin{align*}
	\frac{1}{2}\dfrac{\rmd}{\rmd t}\|\rho\|_{L^2}^2 = -D\|\nabla \rho\|_{L^2}^2 - \frac{\sigma}{\epsilon}\|\rho\|_{L^2} \leq 0, 
\end{align*}
via testing \eqref{PDE2} by $\rho$. This fact enables us to control the error term of the charge density first through its diffusion term, and in turn we estimate the electric potential $\phi$ by establishing an elliptic equation.
By employing the mathematical induction method, we ultimately obtain the optimal error estimate both in time and space for the first time. Moreover, this convergence result can directly be extended to the second-order pressure decoupled method \cite{Wang2022} and ``zero-energy-contribution'' method \cite{Ma2025}, by using the analytical techniques therein.

The rest of the paper is organized as follows. In Section \ref{sec-main}, we present the numerical scheme and the main results of this work, including the optimal convergence rates, unconditional energy stability, unique solvability, and charge conservation. In Section \ref{sec-error}, we provide the rigorous error analysis for the proposed scheme. Several numerical examples are illustrated in Section \ref{sec-numerical} to demonstrate the validity of the theoretical results.

\section{Main results}\label{sec-main}

In this section, we will present the fully discrete scheme of the EHD system \eqref{PDE1}-\eqref{PDE4} together with the main results of energy stability and optimal convergence rates.

\subsection{Variational formulation}

For any integer $k\geq0$ and real number $1\leq p \leq \infty$, we adopt the conventional Sobolev spaces $W^{k,p}(\Omega)$ and, as general, denote by  $H^k(\Omega):=W^{k,2}(\Omega)$ and $L^p(\Omega):=W^{0,p}(\Omega)$. The corresponding vector-valued spaces are denoted by ${\mathbf W}^{k,p}(\Omega) :=[W^{k,p}(\Omega)]^d$ $(d=2,3)$, as well as $\H^k(\Omega):={\mathbf W}^{k,2}(\Omega)$ and $\L^p:={\mathbf W}^{0, p}(\Omega)$.
In the later analysis, we denote by $\big(\cdot,\cdot\big)$ the standard $L^2$ inner product and by $\|\cdot\|_{W^{k,p}}, \|\cdot\|_{H^k}, \|\cdot\|_{L^p}$ the norm of the corresponding (both scalar- and vector- valued) spaces, respectively. 
Moreover, we need the following admissible spaces for $(\u, p)$
\begin{align*}
	\H^1_0(\Omega):= &\{ \v\in\H^1(\Omega):~ \v=\0 ~ \mbox{on} ~ \partial\Omega \},	\\
	L^2_0(\Omega):= &\{ v\in L^2(\Omega):~ \int_{\Omega} v \rmd x = 0 \}. 
\end{align*}

With these notations, it is seen that the exact solution $(\phi, \rho, \u, p)$ of the EHD system \eqref{PDE1}-\eqref{PDE4} satisfies 
\begin{align}
	& \epsilon\big(\nabla\phi, \nabla\varphi\big) = \big(\rho, \varphi\big),	\label{var1}		\\
	& \big(\partial_t\rho, \chi\big) - \big(\rho\u, \nabla\chi\big) + D\big(\nabla\rho, \nabla\chi\big) 
		+ \frac{\sigma}{\epsilon}\big(\rho, \chi\big) = 0, \label{var2} \\
	& \big(\partial_t\u, \v\big) + b(\u, \u, \v) + \eta\big(\nabla\u, \nabla\v\big) - \big(p, \nabla\cdot\v\big) 
		+ \big(\rho\nabla\phi, \v\big) = 0,	\label{var3}	\\
	& \big(\nabla\cdot\u, q\big) = 0, \label{var4} 
\end{align}
for all $(\varphi, \chi, \v, q) \in (H^1(\Omega), H^1(\Omega), \H^1_0(\Omega), 	L^2_0(\Omega))$ and every $t>0$.
By using $\nabla\cdot\u=0$, we define the following trilinear term \cite{Temam2001}
\begin{align}\label{trilinear}
	b\big(\u, \v, \w\big) = 
	\frac12\big(\u\cdot\nabla\v, \w\big)
	-\frac12\big(\u\cdot\nabla\w, \v\big), \qquad \mbox{for}~ \u, \v, \w\in\H^1_0(\Omega), 
\end{align}
and it is easy to verify $b(\u, \v, \v) =0$.

\begin{thm}
	The variational formulation \eqref{var1}-\eqref{var4} is energy-stable and satisfies the following energy inequality
	\begin{align*}
		\frac{\rmd }{\rmd t} \Big(\epsilon\|\nabla\phi\|_{L^2}^2 + \|\u\|_{L^2}^2\Big) \leq 0.
	\end{align*}

\end{thm}
\begin{proof}
	Taking $\chi=\phi$ in \eqref{var2} yields
	\begin{align}\label{PDE-energy1}
		\big(\partial_t\rho, \phi\big) - \big(\rho\u, \nabla\phi\big) + D\big(\nabla\rho, \nabla\phi\big) 
		+ \frac{\sigma}{\epsilon}\big(\rho, \phi\big)=0.
	\end{align}
	Next, by taking $\varphi=\dfrac{D}{\epsilon}\rho$ and $\varphi = \dfrac{\sigma}{\epsilon}\phi$ in \eqref{var1}, we obtain
	\begin{align}
		D\big(\nabla\phi, \nabla\rho\big) &= \frac{D}{\epsilon}\|\rho\|_{L^2}^2,	\\
		\frac{\sigma}{\epsilon}\big(\rho, \phi\big) &= \sigma\|\nabla\phi\|_{L^2}^2,	
	\end{align} 
	respectively. Then, differentiating \eqref{PDE1} with respect to $t$, by testing $\phi$ we have
	\begin{align}
			\big(\partial_t\rho, \phi\big) = \epsilon\big(\partial_t\nabla\phi, \nabla\phi\big).
	\end{align}
	Letting $\v=\u$ in \eqref{var3}, we arrive at 
	\begin{align}\label{PDE-energy5}
		\big(\partial_t\u, \u\big) + \eta\|\nabla\u\|_{L^2}^2 + \big(\rho\nabla\phi, \u\big) = 0.
	\end{align}
	A combination of \eqref{PDE-energy1}-\eqref{PDE-energy5} leads to 
	\begin{align*}
		\epsilon\big(\partial_t\nabla\phi, \nabla\phi\big) + \frac{D}{\epsilon}\|\rho\|_{L^2}^2	
		+ \sigma\|\nabla\phi\|_{L^2}^2 + \big(\partial_t\u, \u\big)  +  \eta\|\nabla\u\|_{L^2}^2  = 0.
	\end{align*}
	Then, we obtain
	\begin{align*}
		\frac12\frac{\rmd}{\rmd t}  \Big(\epsilon\|\nabla\phi\|_{L^2}^2 + \|\u\|_{L^2}^2\Big)
		= - \Big(\frac{D}{\epsilon}\|\rho\|_{L^2}^2	+ \sigma\|\nabla\phi\|_{L^2}^2 + \eta\|\nabla\u\|_{L^2}^2  \Big) \leq 0, 
	\end{align*}
	i.e., the assertion.
\end{proof}

\subsection{Numerical schemes}

We first present the finite element spaces. Let $\mathcal{T}_h$ be a quasi-uniform partition of domain $\Omega$ into triangles $K_j$ in $\mathbb{R}^2$ (tetrahedra in $\mathbb{R}^3$), $j=1,2,\cdots,N_x$ with the mesh size $h=\max_{j}{\rm diam}K_j$. Then for any integer $r\geq2$, we define the finite element spaces as follow:
\begin{align*}
	\mathcal{H}^r&:=\{v_h\in C(\Omega):\, v_h|_{K_j} \in P_r(K_j), \ \ \forall K_j\in\mathcal{T}_h\},	\\
	\mathcal{U}^r&:=\{\v_h\in\H_0^1(\Omega): \, \v_h|_{K_j} \in [P_r(K_j)]^d, \ \ \forall K_j\in\mathcal{T}_h\},	\\
	\mathcal{P}^{r-1}&:=\{q_h\in L_0^2(\Omega):\, q_h|_{K_j} \in P_{r-1}(K_j), \ \ \forall K_j\in\mathcal{T}_h\},
\end{align*}
where $P_r(K_j)$ denotes the polynomial space of degree $r$ on $K_j$. Here, we utilize the Taylor--Hood element spaces $(\mathcal{U}, \mathcal{P})$ to solve the velocity-pressure pair $(\u, p)$, which obviously satisfies the discrete inf-sup condition
\begin{align*}
	\inf_{0\neq q_h\in\mathcal{P}^{r-1}} \sup_{{\bf 0}\neq {\v}_h\in \mathcal{U}^r} 
	\frac{\big(q_h, \nabla\cdot\v_h\big)}{\|\v_h\|_{\H^1(\Omega)}}
	\geq C,
\end{align*}
for a positive constant $C$.

For the time discretization, we adopt the uniform partition of temporal interval $[0, T]$ as $\tau = \frac{T}{N_T}$ for a fixed positive integer $N_T$, and define $t^n = n\tau, n=0,1,...,N_T$.
As usual, we denote by $v^n:=v(\cdot, t^n)$ the value at $t^n$ for any function $v$.
Then, the fully discrete scheme of variational formulation \eqref{var1}-\eqref{var4} is to find $(\phi_h^{n+1}, \rho_h^{n+1}, \u_h^{n+1}, p_h^{n+1}) \in (\mathcal{H}^r \times \mathcal{H}^r \times \mathcal{U}^r \times \mathcal{P}^{r-1})$ such that it satisfies 
\begin{align}
	&\epsilon\big(\nabla\phi_h^{n+1}, \nabla\varphi_h\big) - \big(\rho_h^{n+1}, \varphi_h\big) =0, \label{scheme1}		\\
	&\big(D_{\tau}\rho_h^{n+1}, \chi_h\big) - \big(\widetilde{\rho}_h^{n+1}\u_h^{n+1}, \nabla\chi_h\big) 
		+ D\big(\nabla\rho_h^{n+1}, \nabla\chi_h\big) + \frac{\sigma}{\epsilon}\big(\rho_h^{n+1}, \chi_h\big) = 0, \label{scheme2} 	\\
	&\big(D_\tau\u_h^{n+1}, \v_h\big) + b\big(\widetilde{\u}_h^{n+1}, \u_h^{n+1}, \v_h\big) 
		+ \eta\big(\nabla\u_h^{n+1}, \nabla\v_h\big) - \big(p_h^{n+1}, \nabla\cdot\v_h\big) 	
		+ \big(\widetilde{\rho}_h^{n+1} \nabla\phi_h^{n+1}, \v_h\big) =0, 	\label{scheme3}		\\
	&\big(\nabla\cdot\u_h^{n+1}, q_h\big) = 0, 	\label{scheme4} 
\end{align}
for $(\varphi_h, \chi_h, \v_h, q_h) \in (\mathcal{H}^r \times \mathcal{H}^r \times \mathcal{U}^r \times \mathcal{P}^{r-1})$ and $n=0,1,..., N_T-1$. We define 
\begin{equation}\label{notation}
	D_\tau v^{n+1} = 
	\left\{
	\begin{array}{ll}
		\dfrac{v^{n+1}-v^n}{\tau}, 		&n=0,	\vspace{0.1in}\\
		\dfrac{3v^{n+1}-4v^n +v^{n-1}}{2\tau}, 		&n\geq 1,
	\end{array}
	\right.
	\quad \mbox{and} \quad
	\widetilde{v}^{n+1} = 
	\left\{
	\begin{array}{ll}
		v^n, 		&n=0,	\vspace{0.1in}\\
		2v^n - v^{n-1}, 		&n\geq 1,
	\end{array}
	\right.
\end{equation}
for any (scalar- or vector- valued) function $v$.
The initial data is set as
\begin{align}
	\rho_h^0 = P_h\rho^0	
	\quad 	\mbox{and}	\quad
	\u_h^0 = \Q_h\u^0,	\label{initial-num}
\end{align}
where $P_h$ and $\Q_h$ are the standard $L^2$ and Ritz projection operators (defined in Section \ref{subsec-projection}), respectively, and $\phi_h^0$ is defined as the solution of 
\begin{align}
	(\nabla\phi_h^0, \nabla\varphi_h) = (\rho_h^0,\varphi_h), \quad \forall \varphi_h\in\mathcal{H}^r.
\end{align}
Finally, we let $\int_{\Omega}\phi_h^n\rmd x=0$ to guarantee the uniqueness for all $n=0,1,\cdots, N_T$.

In this paper, we assume that the EHD system \eqref{PDE1}-\eqref{PDE4} admits a unique solution $(\phi, \rho, \u, p)$ satisfying 
\begin{align}
	&\|\phi\|_{L^\infty(0,T;H^{r+1})} + \|\rho\|_{L^\infty(0,T;H^{r+1})} 
		+ \|\partial_t\rho\|_{L^\infty(0,T;H^{r+1})} + \|\partial_{tt}\rho\|_{L^\infty(0,T;L^2)}  
		+ \|\partial_{ttt}\rho\|_{L^\infty(0,T;L^2)}		\nonumber \\ 
	&\hspace{0.0in}
		+ \|\u\|_{L^\infty(0,T;{W}^{r+1, 3})} + \|\partial_t\u\|_{L^\infty(0,T;H^{r+1})} 	+ \|\partial_{tt}\u\|_{L^\infty(0,T;L^2)} + \|\partial_{ttt}\u\|_{L^\infty(0,T;L^2)}		
		+ \|p\|_{L^{\infty}(0, T;H^r)}	\nonumber \\ 
	&\hspace{0.0in}
	\leq K.		\label{regularity-ass}
\end{align}
We now present the main result of this paper in the following theorem, and the detailed proof will be shown in Section \ref{sec-error}.
\begin{thm}\label{thm-error}
	Assume that the EHD system \eqref{PDE1}-\eqref{PDE4} has the unique solution $(\phi, \rho, \u, p)$ satisfying the regularity assumption \eqref{regularity-ass}. Then, the discrete system \eqref{scheme1}-\eqref{scheme4} admits a unique solution $(\phi_h^n, \rho_h^n, \u_h^n, p_h^n)$ for $n=1,2,\cdots,N_T$. 
	Moreover, there exist some small positive constants $\tau_0$ and $h_0$, for $\tau\leq \tau_0$, $\tau\leq h^\frac12$, and $h< h_0$, it holds
	\begin{align*}
		\max_{1\leq n\leq N_T}\Big( \|\phi^n-\phi_h^n\|_{L^2} +  \|\rho^n-\rho_h^n\|_{L^2}
			+ \|\u^n-\u_h^n\|_{L^2}  \Big) 
			\leq C(\tau^2 + h^{r+1}),
	\end{align*}
where $C$ is a positive constant independent of $h$ and $\tau$.
\end{thm}

\subsection{Energy stability and unique solvability}

In this subsection we will provide the unconditional energy stability and unique solvability for the fully discrete scheme \eqref{scheme1}-\eqref{scheme4}.

\begin{thm}\label{thm-energy}
	The numerical scheme \eqref{scheme1}-\eqref{scheme4} is uniquely solvable and unconditionally energy-stable, which satisfies the discrete energy dissipation property
	\begin{align*}
		\mathcal{E}_h^{n+1} - \mathcal{E}_h^n \leq 0,
	\end{align*} 
	for $n=1,2, ..., N_T-1$ with
	\begin{align*}
		\mathcal{E}_h^{n+1} := \epsilon\|\nabla\phi_h^{n+1}\|_{L^2}^2 + \epsilon\|2\nabla\phi_h^{n+1}-\nabla\phi_h^n\|_{L^2}^2 
			+ \|\u_h^{n+1}\|_{L^2}^2 + \|2\u_h^{n+1}-\u_h^n\|_{L^2}^2.
	\end{align*} 
	Moreover, the discrete charge density $\rho_h^n$ satisfies $\int_\Omega\rho_h^n \rmd x=0$ for $n=1,2, \cdots, N_T$.
\end{thm}

\begin{proof}
	It is direct to obtain the discrete charge conservation law 
	\begin{align*}
		\big(\rho_h^{n+1}, 1\big) = \epsilon\big(\nabla\phi_h^{n+1}, \nabla 1\big) = 0,
	\end{align*} 
	by taking $\varphi_h=1$ in \eqref{scheme1}.
	
	Next, we prove the unconditional energy stability. By taking $\chi_h = \phi_h^{n+1}$, $\v_h=\u_h^{n+1}$ and $q_h = p_h^{n+1}$ in \eqref{scheme2}-\eqref{scheme4}, respectively, we arrive at
	\begin{align}
		\big(D_{\tau}\rho_h^{n+1}, \phi_h^{n+1}\big) - \big(\widetilde{\rho}_h^{n+1}\u_h^{n+1}, \nabla\phi_h^{n+1}\big) 
			+ D\big(\nabla\rho_h^{n+1}, \nabla\phi_h^{n+1}\big) 
			+ \frac{\sigma}{\epsilon}\big(\rho_h^{n+1}, \phi_h^{n+1}\big) &= 0, 	\label{num-ener1}\\
		\big(D_\tau\u_h^{n+1}, \u_h^{n+1}\big) + \eta\|\nabla\u_h^{n+1}\|_{L^2}^2 
			+ (\widetilde{\rho}_h^{n+1} \nabla\phi_h^{n+1}, \u_h^{n+1}) 	&=0.		\label{num-ener2}
	\end{align}
	In \eqref{scheme1}, we take $\varphi_h = \dfrac{D}{\epsilon} \rho_h^{n+1}$ and $\dfrac{\sigma}{\epsilon}\phi_h^{n+1}$, respectively, and then obtain
	\begin{align}
		D\big(\nabla\phi_h^{n+1}, \nabla\rho_h^{n+1}\big) &= \frac{D}{\epsilon}\|\rho_h^{n+1}\|_{L^2}^2, 		\label{num-ener3}	\\
		\frac{\sigma}{\epsilon}\big(\rho_h^{n+1}, \phi_h^{n+1}\big) &= \sigma\|\nabla\phi_h^{n+1}\|_{L^2}^2.			\label{num-ener4}	
	\end{align}	
	Further, at $t^{n+1}$, $t^n$ and $t^{n-1}$, we all choose $\varphi_h = \phi_h^{n+1}$ and then have
	\begin{align*}
		\epsilon\big(\nabla\phi_h^{n+1}, \nabla\phi_h^{n+1}\big) &= \big(\rho_h^{n}, \phi_h^{n+1}\big),		\\
		\epsilon\big(\nabla\phi_h^{n}, \nabla\phi_h^{n+1}\big) &= \big(\rho_h^{n}, \phi_h^{n+1}\big),		\\
		\epsilon\big(\nabla\phi_h^{n-1}, \nabla\phi_h^{n+1}\big) &= \big(\rho_h^{n-1}, \phi_h^{n+1}\big),	
	\end{align*}
	which implies
	\begin{align}
		(D_\tau\rho_h^{n+1}, \phi_h^{n+1}) = \epsilon(D_\tau\nabla\phi_h^{n+1}, \nabla\phi_h^{n+1}).	\label{num-ener5}
	\end{align}
	
	A combination of \eqref{num-ener1}-\eqref{num-ener5} leads to 
	\begin{align*}
		\epsilon\big(D_\tau\nabla\phi_h^{n+1}, \nabla\phi_h^{n+1}\big) 
			+ \frac{D}{\epsilon}\|\rho_h^{n+1}\|_{L^2}^2 + \sigma\|\nabla\phi_h^{n+1}\|_{L^2}^2
		+\big(D_\tau\u_h^{n+1}, \u_h^{n+1}\big) + \eta\|\nabla\u_h^{n+1}\|_{L^2}^2 = 0.
	\end{align*}
	With the help of the identities
	\begin{align*}
		(a-b)a &= \frac12[a^2-b^2+(a-b)^2],	\\
		(3a-4b+c)a &= \frac12[a^2-b^2 + (2a-b)^2 - (2b-c)^2 + (a-2b+c)^2],
	\end{align*}
	for $n=0$ we obtain 
	\begin{align*}
		\frac1{2\tau}\Big(\epsilon\|\nabla\phi_h^1\|_{L^2}^2 - \|\nabla\phi_h^0\|_{L^2} 
			+ \|\u_h^1\|_{L^2}^2 - \|\u_h^0\|_{L^2}^2\Big) \leq 0,
	\end{align*}
	which gives the boundedness of numerical solution $\nabla\phi_h^1$ and $\u_h^1$.
	For $n\geq 1$, we have
	\begin{align*}
		\frac1{4\tau}\Big(&\epsilon\|\nabla\phi_h^{n+1}\|_{L^2}^2 - \epsilon\|\nabla\phi_h^{}\|_{L^2}^2  
			+ \epsilon\|2\nabla\phi_h^{n+1}-\nabla\phi_h^n\|_{L^2}^2 - \epsilon\|2\nabla\phi_h^{n}-\nabla\phi_h^{n-1}\|_{L^2}^2 	\\
		&+ \|\u_h^{n+1}\|_{L^2}^2 - \|\u_h^{n}\|_{L^2}^2 
			+ \|2\u_h^{n+1}-\u_h^n\|_{L^2}^2 -  \|2\u_h^{n}-\u_h^{n-1}\|_{L^2}^2 \Big) \leq 0,
	\end{align*}
	i.e., the energy stability $\mathcal{E}_h^{n+1} - \mathcal{E}_h^n \leq 0$, $n=1,2,...,N_T-1$.
	
	Finally, the unconditional energy stability implies that the homogeneous linear equation group \eqref{scheme1}-\eqref{scheme4} only admits the zero solution, which also yields the unique solvability. This completes the proof of Theorem \ref{thm-energy}.
\end{proof}

\section{Optimal error estimates}\label{sec-error}

In this section, we will present the proof of Theorem \ref{thm-error}, where the unique solvability has been proven in Theorem \ref{thm-energy}, and thus it remains to prove 
the optimal convergence order.
Throughout this paper, we denote by $C$ a generic positive constant independent of $\tau$ and $h$, which could be different at different occurrences.

\subsection{Preliminaries}\label{subsec-projection}

We first introduce some essential projections together with their properties.
The standard $L^2$ projection $P_h: L^2(\Omega)\rightarrow \mathcal{H}^r$ is defined as: for $v\in L^2(\Omega)$
\begin{align}\label{L2pro}
	\big(v-P_hv, \varphi_h\big) = 0, 	\qquad 	\forall \varphi_h\in\mathcal{H}^r.
\end{align}
The classic Ritz projection $R_h: H^1(\Omega)\rightarrow \mathcal{H}^r$ is: for $v\in H^1(\Omega)$
\begin{align}\label{Ritzpro}
	\big(\nabla(v- R_hv), \nabla\varphi_h\big) = 0, \qquad \forall  \varphi_h\in\mathcal{H}^r,
\end{align}
with $\int_{\Omega}(v-R_hv)\rmd x=0$.
The Stokes projection is denoted by $(\Q_h, Q_h): \H^1_0(\Omega)\times L^2_0(\Omega) \rightarrow \mathcal{U}^r \times \mathcal{P}^{r-1}$, defined as
\begin{align}\label{Stokepro}
	\eta\big(\nabla(\u - \Q_h\u), \nabla\v\big) - \big(p-Q_h p, \nabla\cdot\v_h\big) &= 0, 	\\
	\big(\nabla\cdot(\u - \Q_h\u), q_h\big) &=0, 
\end{align}
for all $(\v_h, q_h) \in \mathcal{U}^r \times \mathcal{P}^{r-1}$.

\begin{lem}[\cite{Brenner2008, Girault1986, Thomee2006, Wheeler1973}]
	For the above projections, the following estimates are recalled:
	\begin{align}
		&\|v - P_hv\|_{L^2} + h\|\nabla(v-P_hv)\|_{L^2} \leq Ch^{r+1}\|v\|_{H^{r+1}},	\label{L2-ets}	\\
		&\|v - R_hv\|_{L^s} + h\|v-R_hv\|_{W^{1,s}} \leq Ch^{r+1} \|v\|_{W^{r+1, s}}, \qquad 2\leq s\leq \infty, 	\label{Ritz-ets}	\\
		&\|\v-\Q_h\v\|_{L^s} + h\|\v-\Q_h\|_{W^{1,s}}  \leq Ch^{r+1}(\|\v\|_{W^{r+1,s}} + \|p\|_{W^{r,s}}), 	\qquad 1<s<\infty \label{Stoke-ets},
	\end{align}
	where $C>0$ is a constant independent of $h$.
\end{lem}

We also need the following inequalities in the later analysis.
\begin{lem}[\cite{Brenner2008}]
	The embedding inequality and the inverse inequality hold:
	\begin{align}
		&\|v\|_{L^l} \leq C\|v\|_{H^1}, \qquad v\in H^1(\Omega) ,	\label{embedding-ine}	\\
		&\|v\|_{W^{m, s}} \leq Ch^{n-m+\frac{d}{s}-\frac{d}{p}}\|v\|_{W^{n,p}}, 
		\qquad v\in \mathcal{H}^r~ \text{or} ~\mathcal{U}^r, \label{inverse-ine}
	\end{align}
	with $1\leq l \leq 6$, $0\leq n\leq m \leq1$, and $1\leq p\leq s \leq \infty$.
	The Poincar\'{e} inequality is
	\begin{equation}
		\begin{array}{l l }
			\|v\|_{L^2} \leq C(\|\nabla v\|_{L^2} + |(v,1)|),	 &\forall v\in H^1(\Omega),	
			\vspace{0.05in}\\
			\|\v\|_{L^2} \leq C\|\nabla\v\|_{L^2},	\qquad &\forall \v\in\H_0^1(\Omega).
		\end{array}
	\end{equation}
\end{lem}

\subsection{Error equations}

By using the projections defined in \eqref{L2pro}-\eqref{Stokepro}, we change the EHD system \eqref{PDE1}-\eqref{PDE4} to the projection form at $t^{n+1}$: 
\begin{align}
	&\epsilon\big(\nabla R_h\phi^{n+1}, \nabla\varphi_h\big) - \big(R_h\rho^{n+1}, \varphi_h\big) = \big(\rho^{n+1}-R_h\rho^{n+1}, \varphi_h\big), \label{projection1}		\\
	&\big(D_{\tau}R_h\rho^{n+1}, \chi_h\big) - \big(\widetilde{\rho}^{n+1}\u^{n+1}, \nabla\chi_h\big) 
			+ D\big(\nabla R_h\rho^{n+1}, \nabla\chi_h\big) + \frac{\sigma}{\epsilon}\big(R_h\rho^{n+1}, \chi_h\big) 	\nonumber	\\
	&\hspace{2in}		
		=  \frac{\sigma}{\epsilon}\big(R_h\rho^{n+1} - \rho^{n+1}, \chi_h\big) + \mathfrak{T}^{n+1}_{\rho}(\chi_h) , \label{projection2} 	\\
	&\big(D_\tau\Q_h\u^{n+1}, \v_h\big) + b\big(\widetilde{\u}^{n+1}, \u^{n+1}, \v_h\big) 
		+ \eta\big(\nabla\Q_h\u^{n+1}, \nabla\v_h\big) - \big(Q_hp^{n+1}, \nabla\cdot\v_h\big) 	\nonumber	\\
	&\hspace{2in}
	+ \big(\widetilde{\rho}^{n+1} \nabla\phi^{n+1}, \v_h\big) = \mathfrak{T}^{n+1}_{u}(\v_h), 	\label{projection3}		\\
	&\big(\nabla\cdot\Q_h\u^{n+1}, q_h\big) = 0, 	\label{projection4} 
\end{align}
 for all $(\varphi_h, \chi_h, \v_h, q_h) \in (\mathcal{H}^r \times \mathcal{H}^r \times \mathcal{U}^r \times \mathcal{P}^{r-1})$ and $n=0,1,..., N_T-1$, where the truncation error terms are calculated as
\begin{align*}
	\mathfrak{T}^{n+1}_{\rho}(\chi_h) 
		&= \big(D_\tau R_h\rho^{n+1} - \partial_t\rho^{n+1}, \chi_h\big) - \big((\widetilde{\rho}^{n+1}- \rho^{n+1})\u^{n+1}, \nabla\chi_h\big),	\\
	\mathfrak{T}^{n+1}_{u}(\v_h) 
		&= \big(D_\tau \Q_h\u^{n+1} - \partial_t\u^{n+1}, \v_h\big) + b\big(\widetilde{\u}^{n+1}-\u^{n+1}, \u^{n+1}, \v_h\big)	\\
		&+ \big( (\widetilde{\rho}^{n+1} - \rho^{n+1})\nabla\phi^{n+1}, \v_h\big).
\end{align*}

By denoting the error functions by
\begin{equation*}
	\begin{array}{l l}
	e_\rho^n = R_h\rho^n - \rho_h^n, 	\qquad 		e_\phi^n = R_h\phi^n - \phi_h^n,	\qquad
	\e_u^n = \Q_h\u^n - \u_h^n,				  \qquad 	  e_p^n = Q_h p^n -p_h^n,
	\end{array}
\end{equation*}
making difference between projection equation \eqref{projection1}-\eqref{projection4} and numerical scheme \eqref{scheme1}-\eqref{scheme4} leads to the following error equation:
\begin{align}
	&
	\epsilon\big(\nabla e_\phi^{n+1}, \nabla\varphi_h\big) - \big(e_\rho^{n+1}, \varphi_h\big) 
		= \big(\rho^{n+1}-R_h\rho^{n+1}, \varphi_h\big), \label{error1}		\\
	&
	\big(D_{\tau}e_\rho^{n+1}, \chi_h\big) - \big(\widetilde{\rho}^{n+1}\u^{n+1} 
		- \widetilde{\rho}_h^{n+1}\u_h^{n+1}, \nabla\chi_h\big) 
		+ D\big(\nabla e_\rho^{n+1}, \nabla\chi_h\big) \nonumber	\\
	&\hspace{0.5in}		
	+ \frac{\sigma}{\epsilon}\big(e_\rho^{n+1}, \chi_h\big) 	
		=  \frac{\sigma}{\epsilon}\big(R_h\rho^{n+1} - \rho^{n+1}, \chi_h\big) + \mathfrak{T}_{\rho}^{n+1}(\chi_h) , \label{error2} 	\\
	&
	\big(D_\tau\e_u^{n+1}, \v_h\big) + \Big( b\big(\widetilde{\u}^{n+1}, \u^{n+1}, \v_h\big) 
		- b\big(\widetilde{\u}_h^{n+1}, \u_h^{n+1}, \v_h\big) \Big)  
		+ \eta\big(\nabla\e_u^{n+1}, \nabla\v_h\big) \nonumber	\\
	&\hspace{0.5in}
	- \big(e_p^{n+1}, \nabla\cdot\v_h\big) 	
		+ \big(\widetilde{\rho}^{n+1} \nabla\phi^{n+1} - \widetilde{\rho}_h^{n+1} \nabla\phi_h^{n+1}, \v_h\big) 
		= \mathfrak{T}_{u}^{n+1}(\v_h), 	\label{error3}		\\
	&\big(\nabla\cdot\e_u^{n+1}, q_h\big) = 0, 	\label{error4} 
\end{align}
for all $(\varphi_h, \chi_h, \v_h, q_h) \in (\mathcal{H}^r \times \mathcal{H}^r \times \mathcal{U}^r \times \mathcal{P}^{r-1})$ and $n=0,1,..., N_T-1$.

\subsection{Estimates for $\mbox{\boldmath $n$}\,{\bf =0}$}
Since the numerical scheme \eqref{scheme1}-\eqref{scheme4} is a multi-step formulation, we need to obtain the error estimate at $t=t^1$.
For the first step $n=0$, we actually  have 
\begin{align}
	&
	\epsilon\big(\nabla e_\phi^1, \nabla\varphi_h\big) - \big(e_\rho^1, \varphi_h\big) 
		= \big(\rho^1-R_h\rho^1, \varphi_h\big), \label{error1-0}		\\
	&
	\frac1\tau\big(e_\rho^1-e_\rho^0, \chi_h\big) - \big(\rho^0\u^1 - \rho_h^0\u_h^1, \nabla\chi_h\big) 
		+ D\big(\nabla e_\rho^1, \nabla\chi_h\big) \nonumber	\\
	&\hspace{0.5in}		
	+ \frac{\sigma}{\epsilon}\big(e_\rho^1, \chi_h\big) 	
		=  \frac{\sigma}{\epsilon}\big(R_h\rho^1 - \rho^1, \chi_h\big) 	
		+ \mathfrak{T}^1_{\rho}(\chi_h) , \label{error2-0} 	\\
	&
	\frac1\tau\big(\e_u^1, \v_h\big) + \Big( b\big(\u^0, \u^1, \v_h\big) - b\big(\u_h^0, \u_h^1, \v_h\big) \Big)  
		+ \eta\big(\nabla\e_u^1, \nabla\v_h\big) \nonumber	\\
	&\hspace{0.5in}
	- \big(e_p^1, \nabla\cdot\v_h\big) 	
		+ \big(\rho^0 \nabla\phi^1 -\rho_h^0\nabla\phi_h^1, \v_h\big) = \mathfrak{T}^1_{u}(\v_h), 	\label{error3-0}		\\
	&\big(\nabla\cdot\e_u^1, q_h\big) = 0, 	\label{error4-0} 
\end{align}
for all $(\varphi_h, \chi_h, \v_h, q_h) \in (\mathcal{H}^r \times \mathcal{H}^r \times \mathcal{U}^r \times \mathcal{P}^{r-1})$. \vspace{0.2in}

{\bf (1) Estimate for $\e_u^1$:} Adopting $\v_h = \e_u^1$ and $q_h = e_p^1$ in \eqref{error3-0} and \eqref{error4-0}, respectively, we arrive at 
\begin{align}
	&\frac1\tau\|\e_u^1\|_{L^2}^2 + \eta\|\nabla\e_u^1\|_{L^2}^2	\nonumber	\\
		&= -\Big(b\big(\u^0, \u^1, \e_u^1\big) - b\big(\u_h^0, \u_h^1, \e_u^1\big)\Big)	
			- \big(\rho^0\nabla\phi^1 - \rho_h^0\nabla\phi_h^1, \e_u^1\big)	 + \mathfrak{T}^1_{u}(\e_u^1)	\nonumber	\\
		&=: \sum_{i=1}^{3} I_{1i}.	\label{I1}
\end{align}

For $I_{11}$, from definition \eqref{trilinear} we have
\begin{align}
	I_{11} =& - \bigg\{ \Big[ \frac12\big(\u^0\cdot\nabla\u^1, \e_u^1\big) - \frac12\big(\u^0\cdot\nabla\e_u^1, \u^1\big) \Big]
				- \Big[ \frac12\big(\u_h^0\cdot\nabla\u_h^1, \e_u^1\big) - \frac12\big( \u_h^0\cdot\nabla\e_u^1, \u_h^1 \big) \Big] \bigg\}	\nonumber \\
			=& -\frac12 \bigg\{ \Big[ \big(\u^0\cdot\nabla\u^1, \e_u^1\big) - \big(\u_h^0\cdot\nabla\u_h^1, \e_u^1\big) \Big]
					- \Big[ \big(\u^0\cdot\nabla\e_u^1, \u^1\big) - \big(\u_h^0\cdot\nabla\e_u^1, \u_h^1\big) \Big] \bigg\}			\nonumber	\\
			=& -\frac12 \bigg\{ \Big[ \big((\u^0-\Q_h\u^0)\cdot\nabla\u^1, \e_u^1\big) + \big(\e_u^0\cdot\nabla\u^1, \e_u^1\big) 
					+ \big(\u_h^0\cdot\nabla(\u^1-\Q_h\u^1), \e_u^1\big) 	\nonumber	\\
				& + \big(\u_h^0\cdot\nabla\e_u^1, \e_u^1\big) \Big] 
					- \Big[ \big((\u^0-\Q_h\u^0)\cdot\nabla\e_u^1, \u^1\big) + \big(\e_u^0\cdot\nabla\e_u^1, \u^1\big)
					+ \big(\u_h^0\cdot\nabla\e_u^1, \u^1-\Q_h\u^1\big) 	\nonumber	\\
				& + \big(\u_h^0\cdot\nabla\e_u^1, \e_u^1\big) \Big]		\bigg\}		\nonumber	\\
			=& -\frac12 \Big[ \big((\u^0-\Q_h\u^0)\cdot\nabla\u^1, \e_u^1\big) + \big(\u_h^0\cdot\nabla(\u^1-\Q_h\u^1), \e_u^1\big)
					- \big((\u^0-\Q_h\u^0)\cdot\nabla\e_u^1, \u^1\big)		\nonumber	\\
				&	- \big(\u_h^0\cdot\nabla\e_u^1, \u^1-\Q_h\u^1\big)  \Big]	
					\qquad	(\mbox{Via definition \eqref{initial-num},}~\e_u^0 = \0)	\nonumber	\\
			\leq& \frac12\Big( \|\u^0 - \Q_h\u^0\|_{L^2}\|\nabla\u^1\|_{L^\infty}\|\e_u^1\|_{L^2}
					+ \|\nabla\cdot\u_h^0\|_{L^2}\|\u^1-\Q_h\u^1\|_{L^3}\|\e_u^1\|_{L^6}	\nonumber	\\
				& + \|\u^0-\Q_h\u^0\|_{L^2}\|\nabla\e_u^1\|_{L^2}\|\u^1\|_{L^\infty}
					+ 2\|\u_h^0\|_{L^6}\|\nabla\e_u^1\|_{L^2}\|\u^1-\Q_h\u^1\|_{L^3}	\Big)	\nonumber	\\
			\leq& C\big(h^{2r+2}+\|\e_u^1\|_{L^2}^2\big) + \frac{\eta}{4}\|\nabla\e_u^1\|_{L^2}^2, 	\label{I11}
\end{align}
where we have used the integration by parts as follows
\begin{align*}
	&\big(\u_h^0\cdot\nabla(\u^1-\Q_h\u^1), \e_u^1\big) 	\\
	&= - \big( \nabla\cdot\u_h^0, (\u^1-\Q_h\u^1)\cdot\e_u^1	\big) - \big(\u_h^0\cdot\nabla\e_u^1, \u^1-\Q_h\u^1\big)	\\
	&\leq \|\nabla\cdot\u_h^0\|_{L^2}\|\u^1-\Q_h\u^1\|_{L^3}\|\e_u^1\|_{L^6} + \|\u_h^0\|_{L^6}\|\nabla\e_u^1\|_{L^2}\|\u^1-\Q_h\u^1\|_{L^3}.
\end{align*}

Term $I_{12}$ can be controlled by
\begin{align}
	I_{12} =& -\Big[ \big((\rho^0 - R_h\rho^0)\nabla\phi^1, \e_u^1\big) + \big(e_\rho^0\nabla\phi^1, \e_u^1\big)
				+ \big(\rho_h^0\nabla(\phi^1-R_h\phi^1), \e_u^1\big) + \big(\rho_h^0\nabla e_\phi^1, \e_u^1\big) \Big] \nonumber	\\
			\leq& \|\rho^0-R_h\rho^0\|_{L^2}\|\nabla\phi^1\|_{L^\infty}\|\e_u^1\|_{L^2} 
				+ \|e_\rho^0\|_{L^2}\|\nabla\phi^1\|_{L^\infty}\|\e_u^1\|_{L^2} 	\nonumber	\\
				&+ \|\nabla\cdot\e_u^1\|_{L^2}\|\rho_h^0\|_{L^\infty}\|\phi^1-R_h\phi^1\|_{L^2}
				+ \|\nabla\rho_h^0\|_{L^4}\|\phi^1-R_h\phi^1\|_{L^2}\|\e_u^1\|_{L^4}	\nonumber	\\
				& + \|\rho_h^0\|_{L^\infty}\|\nabla e_\phi^1\|_{L^2}\|\e_u^1\|_{L^2}	\nonumber	\\
			\leq& C\Big(h^{r+1}\|\e_u^1\|_{L^2} + \|e_\rho^0\|_{L^2}\|\e_u^1\|_{L^2} + h^{r+1}\|\nabla\cdot\e_u^1\|_{L^2}
				+ h^{r+1}\|\e_u^1\|_{L^4} + \|\nabla e_\phi^1\|_{L^2}\|\e_u^1\|_{L^2}\Big)		\nonumber	\\
			\leq& C\Big(h^{2r+2} + \|\e_u^1\|_{L^2}^2 + \|\nabla e_\phi^1\|_{L^2}^2 \Big) 
				+ \frac{\eta}{4}\|\nabla\e_u^1\|_{L^2}^2,	\label{I12}
\end{align}
with $\|e_\rho^0\|_{L^2} = \|P_h\rho^0-\rho_h^0\|_{L^2}\leq Ch^{r+1}$, where we have used 
\begin{align*}
	&\big(\rho_h^0\nabla(\phi^1-R_h\phi^1), \e_u^1\big) 	\\
	&= - \big(\nabla\cdot\e_u^1, \rho_h^0(\phi^1-R_h\phi^1)\big) - \big(\nabla\rho_h^0(\phi^1-R_h\phi^1), \e_u^1\big) \\
	&\leq	\|\nabla\cdot\e_u^1\|_{L^2}\|\rho_h^0\|_{L^\infty}\|\phi^1-R_h\phi^1\|_{L^2}
		+ \|\nabla\rho_h^0\|_{L^4}\|\phi^1-R_h\phi^1\|_{L^2}\|\e_u^1\|_{L^4},
\end{align*}
and by using the Sobolev inequality, for $r\geq2$ it holds
\begin{align*}
	\|\rho_h^0\|_{L^\infty} 
	\leq& \|\rho_h^0\|_{W^{1,4}} = \|P_h\rho^0\|_{W^{1,4}}	\\
	\leq& \|P_h\rho^0 - R_h\rho^0\|_{W^{1,4}} + \|R_h\rho^0\|_{W^{1,4}} \\
	\leq& Ch^{-\frac d4}\|P_h\rho^0 - R_h\rho^0\|_{H^1} + C		\\
	\leq& Ch^{-\frac d4}\big(\|P_h\rho^0 - \rho^0\|_{H^1}  + \|\rho^0 - R_h\rho^0\|_{H^1} \big) +C	\\
	\leq& Ch^{-\frac d4}Ch^r + C	\\
	\leq& C+1.
\end{align*}

The truncated error term $I_{13}$ can directly be estimated as 
\begin{align}
	I_{13} =&
		\big(\frac1\tau(\Q_h\u^1-\Q_h\u^0) - \partial_t\u^1, \e_u^1\big) + b\big(\u^0-\u^1, \u^1, \e_u^1\big)
	+ \big( (\rho^0 - \rho^1)\nabla\phi^1, \e_u^1\big)	\nonumber	\\
	=& \frac1\tau\big(\Q_h(\u^1-\u^0) - (\u^1-\u^0), \e_u^1\big) + \big(\frac{\u^1-\u^0}{\tau}-\partial_t\u^1, \e_u^1\big)	\nonumber \\
		&+ \frac12\big((\u^0-\u^1)\cdot\nabla\u^1, \e_u^1\big) - \frac12\big((\u^0-\u^1)\cdot\nabla\e_u^1, \u^1\big)	
		+ \big((\rho^0-\rho^1)\nabla\phi^1, \e_u^1\big)	\nonumber	\\	
	\leq& C\big(h^{2r+2} + \|\e_u^1\|_{L^2}^2\big) + C\tau\|\e_u^1\|_{L^2}^2 + C\tau\|\nabla\e_u^1\|_{L^2}^2 .	\label{I13}
\end{align}

Consequently, substituting \eqref{I11}-\eqref{I13} into \eqref{I1} yields
\begin{align*}
	\frac1\tau\|\e_u^1\|_{L^2}^2 + \eta\|\nabla\e_u^1\|_{L^2}^2 
		\leq C\big(h^{2r+2} + \|\e_u^1\|_{L^2}^2+ \|\nabla e_\phi^1\|_{L^2}^2 \big) 
			+ \big(\frac{\eta}2+C\tau\big)\|\nabla\e_u^1\|_{L^2}^2 + C\tau\|\e_u^1\|_{L^2}, 
\end{align*}
and then, for some $\tau$ small enough we arrive at
\begin{align*}
	\|\e_u^1\|_{L^2}^2 + \frac{\eta\tau}{4}\|\nabla\e_u^1\|_{L^2}^2
	\leq& C\tau h^{2r+2} + C\tau\|\e_u^1\|_{L^2}^2 + C\tau\|\nabla e_\phi^1\|_{L^2}^2
		+ C\tau^2\|\e_u^1\|_{L^2}		\nonumber	\\
	\leq& C\tau h^{2r+2} + C\tau\|\e_u^1\|_{L^2}^2 + C\tau\|\nabla e_\phi^1\|_{L^2}^2
		+ C\tau^4 + \frac14\|\e_u^1\|_{L^2}^2.
\end{align*}
which in turn yields 
\begin{align}\label{I1-final}
	\frac12\|\e_u^1\|_{L^2}^2 + \frac{\eta\tau}{4}\|\nabla\e_u^1\|_{L^2}^2 \leq C \big(\tau h^{2r+2} + \tau^4\big) 
	 	+ C\tau\|\nabla e_\phi^1\|_{L^2}^2.
\end{align}
\vspace{0.0in}

{\bf (2) Estimate for $e_\rho^1$:} Taking $\chi_h=e_\rho^1$ in \eqref{error2-0} gives 
\begin{align}\label{I2}
	&\frac{1}{2\tau}\Big(\|e_\rho^1\|_{L^2}^2 - \|e_\rho^0\|_{L^2}^2 + \|e_\rho^1-e_\rho^0\|_{L^2}^2\Big)
		+ D\|\nabla e_\rho^1\|_{L^2}^2 + \frac{\sigma}{\epsilon}\|e_\rho^1\|_{L^2}^2	\nonumber	\\
	&=
		\big(\rho^0\u^1-\rho_h^0\u_h^1, \nabla e_\rho^1\big) + \frac{\sigma}{\epsilon}\big(R_h\rho^1-\rho^1, e_\rho^1\big)
			+ \mathfrak{T}_\rho^1(e_\rho^1)		\nonumber	\\
	&:= \sum_{i=1}^{3}I_{2i}.
\end{align}

It is direct to estimate these terms as
\begin{align}
	I_{21}=&\big((\rho^0-R_h\rho^0)\u^1, \nabla e_\rho^1\big) + \big( e_\rho^0\u^1, \nabla e_\rho^1\big)
		+ \big(\rho_h^0(\u^1-\Q_h\u^1), \nabla e_\rho^1\big) + \big(\rho_h^0\e_u^1, \nabla e_\rho^1\big) 	\nonumber	\\
		\leq&
			\|\rho^0-R_h\rho^0\|_{L^2}\|\u^1\|_{L^\infty}\|\nabla e_\rho^1\|_{L^2} 
				+ \|e_\rho^0\|_{L^2}\|\u^1\|_{L^\infty}\|\nabla e_\rho^1\|_{L^2}	\nonumber	\\
			&
			+ \|\rho_h^0\|_{L^\infty}\|\u^1-\Q_h\u^1\|_{L^2}\|\nabla e_\rho^1\|_{L^2}
			+ \|\rho_h^0\|_{L^\infty}\|\e_u^1\|_{L^2}\|\nabla e_\rho^1\|_{L^2}	\nonumber	\\
		\leq&
			C\big(h^{2r+2} + \|\e_u^1\|_{L^2}^2\big) + \frac{D}{4}\|\nabla e_\rho^1\|_{L^2}^2,	\label{I21}	\\
	I_{22}\leq& 
			\frac{\sigma}{\epsilon}\|R_h\rho^1-\rho^1\|_{L^2}\|e_\rho^1\|_{L^2} \leq Ch^{2r+2} 
				+ \frac{\sigma}{4\epsilon}\|e_\rho^1\|_{L^2}^2, 		\label{I22}			\\
	I_{23}=& 
			\frac{1}{\tau}\big(R_h\rho^1-R_h\rho^0 - (\rho^1-\rho^0), e_\rho^1\big)
				+ \big(\frac{1}{\tau}(\rho^1-\rho^0) - \partial_t\rho^1, e_\rho^1\big) 
				+ \big((\rho^0-\rho^1)\u^1, \nabla e_\rho^1\big)		\nonumber	\\
			\leq&
				Ch^{2r+2} + \frac{\sigma}{4\epsilon}\|e_\rho^1\|_{L^2}^2 + C\tau\|e_\rho^1\|_{L^2} 
					+ C\tau\|\nabla e_\rho^1\|_{L^2}.	\label{I23}
\end{align}

Then, by using \eqref{I21}-\eqref{I23} we can rewrite \eqref{I2} as 
\begin{align*}
	&\frac{1}{2\tau}\Big(\|e_\rho^1\|_{L^2}^2 - \|e_\rho^0\|_{L^2}^2 + \|e_\rho^1-e_\rho^0\|_{L^2}^2\Big)
		+ D\|\nabla e_\rho^1\|_{L^2}^2 + \frac{\sigma}{\epsilon}\|e_\rho^1\|_{L^2}^2	\\
	&\leq
		C\big(h^{2r+2} + \|\e_u^1\|_{L^2}^2\big) + \frac{D}{4}\|\nabla e_\rho^1\|_{L^2}^2 
			+ \frac{\sigma}{2\epsilon}\|e_\rho^1\|_{L^2}^2 + C\tau\|e_\rho^1\|_{L^2}
			+ C\tau\|\nabla e_\rho^1\|_{L^2}	\\
	&=
		C\big(h^{2r+2} + \|\e_u^1\|_{L^2}^2\big) 
			+ (\frac{D}{4} + C\tau) \|\nabla e_\rho^1\|_{L^2}^2 
			+ (\frac{\sigma}{2\epsilon} +  C\tau) \|e_\rho^1\|_{L^2}^2,
\end{align*}
which in turn leads to
\begin{align}\label{I2-final}
	\frac12\|e_\rho^1\|_{L^2} + \frac D2\tau\|\nabla e_\rho^1\|_{L^2}^2 + \frac{\sigma}{4\epsilon}\tau\|e_\rho^1\|_{L^2}^2
		\leq Ch^{2r+2} + C\tau h^{2r+2} + C\tau\|\e_u^1\|_{L^2}^2,	
\end{align}
for some small $\tau$ satisfying $C\tau \leq \min\{\frac{D}{4}, \frac{\sigma}{4\epsilon}\}$. \vspace{0.2in}

{\bf (3) Estimate for $e_\phi^1$:} Adopting $\varphi_h = e_\phi^1$ \eqref{error1-0}, we obtain
\begin{align*}
	\epsilon\|\nabla e_\phi^1\|_{L^2}^2 
		=& \big(e_\rho^1, e_\phi^1\big) + \big(\rho^1-R_h\rho^1, e_\phi^1\big)		\\
		\leq&
			\big(\|e_\rho^1\|_{L^2} + \|\rho^1-R_h\rho^1\|_{L^2}\big)\|e_\phi^1\|_{L^2} \\
		\leq&
			C\big(\|e_\rho^1\|_{L^2} + h^{r+1}\big) \|\nabla e_\phi^1\|_{L^2}	\\
		\leq&
			C\big(\|e_\rho^1\|_{L^2}^2 + h^{2r+2}\big) + \frac{\epsilon}{2}\|\nabla e_\phi^1\|_{L^2}^2, 
\end{align*}
which implies
\begin{align}
	\|\nabla e_\phi^1\|_{L^2}^2  \leq C\big(\|e_\rho^1\|_{L^2}^2 + h^{2r+2}\big).	\label{I3-final}
\end{align}
\vspace{0.0in}

{\bf (4) Final estimate at $t=t^1$:} By combining \eqref{I1-final}, \eqref{I2-final} and \eqref{I3-final} we arrive at
\begin{align*}
	&\frac12\|\e_u^1\|_{L^2}^2 + \frac12\|e_\rho^1\|_{L^2}^2 + \tau\Big(\frac{\eta}{4}\|\nabla\e_u^1\|_{L^2}^2 
		+ \frac{D}{2}\|\nabla e_\rho^1\|_{L^2}^2 + \frac{\sigma}{4\epsilon}\|e_\rho^1\|_{L^2}^2\Big)		\\
	&\leq
		C\big(\tau h^{2r+2} + h^{2r+2} + \tau^4\big) + C\tau\|\e_u^1\|_{L^2}^2 + C\tau\|e_\rho^1\|_{L^2}^2, 
\end{align*}
which leads to 
\begin{align*}
	\|\e_u^1\|_{L^2}^2 + \|e_\rho^1\|_{L^2}^2 \leq C\big(h^{2r+2} + \tau^4\big),
\end{align*}
for some small $\tau\leq \frac{1}{4C}$, and furthermore,
\begin{align*}
	\|\nabla e_\phi^1\|_{L^2}^2 \leq C\big(h^{2r+2} + \tau^4\big) 
\end{align*}
by using \eqref{I3-final}.

\subsection{Estimates for $\mbox{\boldmath $n$}\,{\bf \geq1}$}

We now proceed with the proof by utilizing the mathematical induction method.
Suppose that for $m\leq n$, it holds 
\begin{align}\label{induction-ass}
	\|\e_u^m\|_{L^2} + \|e_\rho^m\|_{L^2} \leq C\big(h^{r+1} + \tau^2\big),
\end{align}
which consequently leads to
\begin{align}
	&\|\nabla\u_h^m\|_{L^2} + \|\nabla \rho_h^m\|_{L^2}	\nonumber	\\
	&\leq
		\|\nabla \e_u^m\|_{L^2} + \|\nabla\Q_h\u^m\|_{L^2} + \|\nabla e_\rho^m\|_{L^2} + \|\nabla R_h\rho^m\|_{L^2}	\nonumber	\\
	&\leq
		Ch^{-1}(\| \e_u^m\|_{L^2} + \|e_\rho^m\|_{L^2}) + C(\|\nabla\u^m\|_{L^2} + \|\nabla\rho^m\|_{L^2})	\nonumber 	\\
	&\leq
		C\big(h^r + h^{-1}\tau^2\big) + C	\nonumber	\\
	&\leq
		C+1,
\end{align}
with the assumption $\tau\leq Ch^{\frac12}$. 
We have proven the validity for $m=0, 1$, and will recover this induction assumption at $m=n+1$.  \vspace{0.2in}

{\bf (5) Estimate for $\e_u^{n+1}$:} We still estimate $\e_u^{n+1}$ first by adopting $\v_h=\e_u^{n+1}$ and $q_h=p_h^{n+1}$ in \eqref{error3}-\eqref{error4}, respectively, and then obtain
\begin{align}
		&\frac1{4\tau}\Big(\|\e_u^{n+1}\|_{L^2}^2 - \|\e_u^n\|_{L^2}^2 
			+ \|2\e_u^{n+1}-\e_u^n\|_{L^2}^2 - \|2\e_u^n-\e_u^{n-1}\|_{L^2}^2	
			+ \|\e_u^{n+1} - 2\e_u^n+\e_u^{n-1}\|_{L^2}^2\Big)					\nonumber	\\
		&\hspace{0.2in} 
		 	+ \eta\|\nabla\e_u^{n+1}\|_{L^2}^2 	
			=- \Big( b(\widetilde{\u}^{n+1}, \u^{n+1}, \e_u^{n+1}) - b(\widetilde{\u}_h^{n+1}, \u_h^{n+1}, \e_u^{n+1}) \Big) 			\nonumber	\\
		&\hspace{0.2in}
			+\Big(- (\widetilde{\rho}^{n+1} \nabla\phi^{n+1} - \widetilde{\rho}_h^{n+1} \nabla\phi_h^{n+1}, \e_u^{n+1})\Big)		
			+\mathfrak{T}_{u}^{n+1}(\e_u^{n+1})			\nonumber	\\
		&:=\sum_{i=1}^3I_{4i}, 	\label{I4}		
\end{align}

For $I_{41}$, it can be estimated by 
\begin{align}
	I_{41}=&-\frac12\bigg\{ \Big[\big(\widetilde{\u}^{n+1}\cdot\nabla\u^{n+1}, \e_u^{n+1}\big) 
				- \big(\widetilde{\u}_h^{n+1}\cdot\nabla\u_h^{n+1}, \e_u^{n+1}\big) \Big]					\nonumber	\\
			& 
				- \Big[ \big(\widetilde{\u}^{n+1}\cdot\e_u^{n+1}, \nabla\u^{n+1}\big) 
				- \big(\widetilde{\u}_h^{n+1}\cdot\e_u^{n+1}\big), \nabla\u_h^{n+1}\Big]	\bigg\}		\nonumber	\\
		=& -\frac12
			\bigg\{ \Big[\big((\widetilde{\u}^{n+1}-\Q_h\widetilde{\u}^{n+1})\cdot\nabla\u^{n+1}, \e_u^{n+1}\big)
				+ \big(\widetilde{\e}_u^{n+1}\cdot\nabla\u^{n+1}, \e_u^{n+1}\big)						\nonumber	\\
		&
				+ \big(\widetilde{\u}_h^{n+1}\cdot\nabla(\u^{n+1}-\Q_h\u^{n+1}), \e_u^{n+1}\big)
				+ \big(\widetilde{\u}_h^{n+1}\cdot\nabla\e_u^{n+1}, \e_u^{n+1}\big)		\Big]		\nonumber	\\
		&
			-\Big[\big((\widetilde{\u}^{n+1}-\Q_h\widetilde{\u}^{n+1})\cdot \nabla\e_u^{n+1}, \u^{n+1}\big)
			+ \big(\widetilde{\e}_u^{n+1}\cdot\nabla\e_u^{n+1}, \u^{n+1}\big)						\nonumber	\\
		&
			+ \big(\widetilde{\u}_h^{n+1}\cdot\nabla \e_u^{n+1}, \u^{n+1}-\Q_h\u^{n+1}\big)
			+ \big(\widetilde{\u}_h^{n+1}\cdot\nabla\e_u^{n+1}, \e_u^{n+1}\big)		\Big]	\bigg\}		\nonumber	\\
		:=&
			-\frac12\bigg\{\sum_{i=1}^{4}I_{41,i} - \sum_{i=5}^{8}I_{41,i}\bigg\}.	\label{I41-ori}
\end{align}
We estimate these terms one by one as 
\begin{align*}
	I_{41,1}\leq& \|\widetilde{\u}^{n+1} - \Q_h\widetilde{\u}^{n+1}\|_{L^2}\|\nabla\u^{n+1}\|_{L^3}\|\e_u^{n+1}\|_{L^6}
				\leq Ch^{2r+2} + \frac{\eta}{8}\|\nabla\e_u^{n+1}\|_{L^2}^2,	\\
	I_{41,2}\leq& \|\widetilde{\e}_u^{n+1}\|_{L^2}\|\nabla\u^{n+1}\|_{L^3}\|\e_u^{n+1}\|_{L^6}
				\leq C\|\widetilde{\e}_u^{n+1}\|_{L^2}^2 + \frac{\eta}{8}\|\nabla\e_u^{n+1}\|_{L^2}^2,	\\
	I_{41,3}=& - \big(\nabla\cdot\widetilde{\u}_h^{n+1}, (\u^{n+1}-\Q_h\u^{n+1})\cdot\e_u^{n+1}\big)
						- \big(\widetilde{\u}_h^{n+1}\cdot\nabla\e_u^{n+1}, \u^{n+1}-\Q_h\u^{n+1}\big) 	\\
				\leq&
					\|\nabla\cdot\widetilde{\u}_h^{n+1}\|_{L^2}\|\u^{n+1}-\Q_h\u^{n+1}\|_{L^3}\|\e_u^{n+1}\|_{L^6}
					+ \|\widetilde{\u}_h^{n+1}\|_{L^6}\|\nabla\e_u^{n+1}\|_{L^2}\|\u^{n+1}-\Q_h\u^{n+1}\|_{L^3}		\\
				\leq&
					Ch^{2r+2} +\frac{\eta}{8}\|\nabla\e_u^{n+1}\|_{L^2}^2,	\\
	I_{41,4} \, + \, &I_{41,8}=0,			\\
	I_{41, 5}\leq& \|\widetilde{\u}^{n+1}-\Q_h\widetilde{\u}^{n+1}\|_{L^2}\|\nabla\e_u^{n+1}\|_{L^2}\|\u^{n+1}\|_{L^\infty}
					\leq Ch^{2r+2} + \frac{\eta}{8}\|\nabla\e_u^{n+1}\|_{L^2}^2,	\\
	I_{41,6}\leq& \|\widetilde{\e}_u^{n+1}\|_{L^2}\|\nabla\e_u^{n+1}\|_{L^2}\|\u^{n+1}\|_{L^\infty}
					\leq C\|\widetilde{\e}_u^{n+1}\|_{L^2}^2 + \frac{\eta}{8}\|\nabla\e_u^{n+1}\|_{L^2}^2,		\\
	I_{41,7}\leq& \|\widetilde{\u}_h^{n+1}\|_{L^6}\|\nabla \e_u^{n+1}\|_{L^2}\|\u^{n+1}-\Q_h\u^{n+1}\|_{L^3}	
					\leq Ch^{2r+2} +  \frac{\eta}{8}\|\nabla\e_u^{n+1}\|_{L^2}^2, 
\end{align*}
with which we can rewrite \eqref{I41-ori} as
\begin{align}\label{I41}
	I_{41} \leq C\big(h^{2r+2} + \|\widetilde{\e}_u^{n+1}\|_{L^2}^2\big) + \frac{3\eta}{8}\|\nabla\e_u^{n+1}\|_{L^2}^2.
\end{align}

For term $I_{42}$, it can be calculated as 
\begin{align}
	I_{42}=& 
				- \Big[\big((\widetilde{\rho}^{n+1}-R_h\widetilde{\rho}^{n+1})\nabla\phi^{n+1}, \e_u^{n+1}\big)
					+ \big(\widetilde{e}_\rho^{n+1}\nabla\phi^{n+1}, \e_u^{n+1}\big)		\nonumber	\\
				& 
				+ \big(\widetilde{\rho}_h^{n+1}\nabla(\phi^{n+1}-R_h\phi^{n+1}), \e_u^{n+1}\big)
					+ \big(\widetilde{\rho}_h^{n+1}\nabla e_\phi^{n+1}, \e_u^{n+1}\big)	\Big]		\nonumber	\\
			\leq& 
				\|\widetilde{\rho}^{n+1}-R_h\widetilde{\rho}^{n+1}\|_{L^2}\|\nabla\phi^{n+1}\|_{L^\infty}\|\e_u^{n+1}\|_{L^2}
					+ \|\widetilde{e}_\rho^{n+1}\|_{L^2}\|\nabla\phi^{n+1}\|_{L^\infty}\|\e_u^{n+1}\|_{L^2}		\nonumber	\\
				& 
				+ C\big(h^{r+1} + \|\widetilde{e}_\rho^{n+1}\|_{L^2}\big)\|\nabla\e_u^{n+1}\|_{L^2}
					+ C\big(h^{1-\frac d6}\|\nabla\e_u^{n+1}\|_{L^2} + \|\e_u^{n+1}\|_{L^2}\big)\|\nabla e_\phi^{n+1}\|_{L^2}	\nonumber	\\
			\leq&
				C\big(h^{2r+2} + \|\widetilde{e}_\rho^{n+1}\|_{L^2}^2 + \|\e_u^{n+1}\|_{L^2}^2\big)
				 	+ \frac{\eta}{32}\|\nabla\e_u^{n+1}\|_{L^2}^2 + Ch^{2-\frac d3}\|\nabla\e_u^{n+1}\|_{L^2}^2
					+ \|\nabla e_\phi^{n+1}\|_{L^2}^2,				\nonumber	\\ 
			\leq 
				&C\big(h^{2r+2} + \|\widetilde{e}_\rho^{n+1}\|_{L^2}^2 + \|\e_u^{n+1}\|_{L^2}^2\big)
				+  \frac{\eta}{16}\|\nabla\e_u^{n+1}\|_{L^2}^2 + \|\nabla e_\phi^{n+1}\|_{L^2}^2,	\label{I42}
\end{align}
with some small $h$ satisfying $Ch^{2-\frac d3}\leq \frac{\eta}{32}$. Here, for $r\geq 2$, the third term of $I_{42}$ can be controlled by
\begin{align*}
		&\big(-\widetilde{\rho}_h^{n+1}\nabla(\phi^{n+1}-R_h\phi^{n+1}), \e_u^{n+1}\big)		\\
	&=
		 \big(\nabla\widetilde{\rho}_h^{n+1}(\phi^{n+1}-R_h\phi^{n+1}), \e_u^{n+1}\big)
		+  \big(\widetilde{\rho}_h^{n+1}(\phi^{n+1}-R_h\phi^{n+1}), \nabla\cdot\e_u^{n+1}\big)	\\
	&=
		\Big[ \big(-\nabla\widetilde{e}_\rho^{n+1}(\phi^{n+1}-R_h\phi^{n+1}), \e_u^{n+1}\big)
		+ \big(\nabla R_h\widetilde{\rho}^{n+1}(\phi^{n+1}-R_h\phi^{n+1}), \e_u^{n+1}\big) \Big]	\\
		& \ \ ~
		+ \Big[ \big(-\widetilde{e}_\rho^{n+1}(\phi^{n+1}-R_h\phi^{n+1}), \nabla\cdot\e_u^{n+1}\big)
		+  \big((R_h\widetilde{\rho}^{n+1}-\widetilde{\rho}^{n+1})(\phi^{n+1}-R_h\phi^{n+1}), \nabla\cdot\e_u^{n+1}\big)	\\
		& \ \ ~
		+  \big(\widetilde{\rho}^{n+1}(\phi^{n+1}-R_h\phi^{n+1}), \nabla\cdot\e_u^{n+1}\big) \Big]		\\
	&\leq
		\|\nabla\widetilde{e}_\rho^{n+1}\|_{L^2}\|\phi^{n+1}-R_h\phi^{n+1}\|_{L^3}\|\e_u^{n+1}\|_{L^6}
		+ \|\nabla R_h\widetilde{\rho}^{n+1}\|_{L^4}\|\phi^{n+1}-R_h\phi^{n+1}\|_{L^2}\|\e_u^{n+1}\|_{L^4}		\\
		& \ \ ~
		+ \|\widetilde{e}_\rho^{n+1}\|_{L^4}\|\phi^{n+1}-R_h\phi^{n+1}\|_{L^4}\|\nabla\cdot\e_u^{n+1}\|_{L^2}
		+ \|R_h\widetilde{\rho}^{n+1}-\widetilde{\rho}^{n+1}\|_{L^4}\|\phi^{n+1}-R_h\phi^{n+1}\|_{L^4}\|\nabla\cdot\e_u^{n+1}\|_{L^2} \\
		& \ \ ~
		+ \|\widetilde{\rho}^{n+1}\|_{L^\infty}\|\phi^{n+1}-R_h\phi^{n+1}\|_{L^2}\|\nabla\cdot\e_u^{n+1}\|_{L^2}	\\
	&\leq
		Ch^{-1}\|\widetilde{e}_\rho^{n+1}\|_{L^2}~h^r\|\nabla\e_u^{n+1}\|_{L^2}
		+ Ch^{r+1}\|\nabla\e_u^{n+1}\|_{L^2}
		+ Ch^{-\frac{d}{4}}\|\widetilde{e}_\rho^{n+1}\|_{L^2}~h^r\|\nabla\e_u^{n+1}\|_{L^2}		\\
		& \ \ ~
		+ Ch^{2r}\|\nabla\e_u^{n+1}\|_{L^2} + Ch^{r+1}\|\nabla\e_u^{n+1}\|_{L^2}	\\
	&\leq
		C\big(h^{r+1} + \|\widetilde{e}_\rho^{n+1}\|_{L^2}\big)\|\nabla\e_u^{n+1}\|_{L^2},
\end{align*}
and the fourth term of $I_{42}$ is estimated as 
\begin{align*}
	&\big(-\widetilde{\rho}_h^{n+1}\nabla e_\phi^{n+1}, \e_u^{n+1}\big)	\\
	&=
		\big(\widetilde{e}_\rho^{n+1}\nabla e_\phi^{n+1}, \e_u^{n+1}\big) 
		+ \big((\widetilde{\rho}^{n+1}-R_h\widetilde{\rho}^{n+1})\nabla e_\phi^{n+1}, \e_u^{n+1}\big)
		- \big(\widetilde{\rho}^{n+1}\nabla e_\phi^{n+1}, \e_u^{n+1}\big)	\\
	&\leq
		\|\widetilde{e}_\rho^{n+1}\|_{L^3}\|\nabla e_\phi^{n+1}\|_{L^2}\|\e_u^{n+1}\|_{L^6}
		+ \|\widetilde{\rho}^{n+1}-R_h\widetilde{\rho}^{n+1}\|_{L^6}\|\nabla e_\phi^{n+1}\|_{L^2}\|\e_u^{n+1}\|_{L^3}	\\
		& \ \ ~
		+ \|\widetilde{\rho}^{n+1}\|_{L^\infty}\|\nabla e_\phi^{n+1}\|_{L^2}\|\e_u^{n+1}\|_{L^2}		\\
	&\leq
		Ch^{-\frac d6}\|\widetilde{e}_\rho^{n+1}\|_{L^2}\|\nabla e_\phi^{n+1}\|_{L^2}\|\nabla \e_u^{n+1}\|_{L^2}
		+ Ch^r\|\nabla e_\phi^{n+1}\|_{L^2}\,h^{-\frac d6}\|\e_u^{n+1}\|_{L^2}		\\
		& \ \ ~
		+ C\|\nabla e_\phi^{n+1}\|_{L^2}\|\e_u^{n+1}\|_{L^2}			\\
	&\leq
		C\big(h^{1-\frac d6}\|\nabla\e_u^{n+1}\|_{L^2} + \|\e_u^{n+1}\|_{L^2}\big)\|\nabla e_\phi^{n+1}\|_{L^2},
\end{align*}
due to $\tau\leq Ch^{\frac12}$ and then $h^{-\frac d6}\|\widetilde{e}_\rho^{n+1}\|_{L^2}\leq h^{-\frac d6}(h^{r+1}+\tau^2) \leq h^{1-\frac d6}$ via \eqref{induction-ass}.

For the truncation term $I_{43}$, we have 
\begin{align}
	I_{43}
		=& \big(D_\tau(\Q_h\u^{n+1} - \u^{n+1}), \e_u^{n+1}\big) 
			+ \big(D_\tau\u^{n+1}-\partial_t\u^{n+1}, \e_u^{n+1}\big)	\nonumber \\
		&
			+\big(b(\widetilde{\u}_h^{n+1}, \u^{n+1}, \e_u^{n+1}) - b(\u^{n+1}, \u^{n+1}, \e_u^{n+1})\big)
			+ \big((\widetilde{\rho}^{n+1}-\rho^{n+1})\nabla\phi^{n+1}, \e_u^{n+1}\big)	\nonumber	\\
		\leq&
			Ch^{r+1}\|\e_u^{n+1}\|_{L^2} + C\tau^2\|\e_u^{n+1}\|_{L^2} + C\|\e_u^{n+1}\|_{L^2}\|\nabla\e_u^{n+1}\|_{L^2}	\nonumber	\\
		\leq& 
			C\big(h^{2r+2} + \tau^4+ \|\e_u^{n+1}\|_{L^2}^2\big) + \frac{\eta}{16}\|\nabla\e_u^{n+1}\|_{L^2},	\label{I43}
\end{align}
where the analysis is similar to \eqref{I13}.

Combining \eqref{I41}-\eqref{I43}, we rewrite \eqref{I4} as  
\begin{align}
	&\frac1{4\tau}\Big(\|\e_u^{n+1}\|_{L^2}^2 - \|\e_u^n\|_{L^2}^2 
	+ \|2\e_u^{n+1}-\e_u^n\|_{L^2}^2 - \|2\e_u^n-\e_u^{n-1}\|_{L^2}^2	
	\Big)				
	+ \frac\eta2\|\nabla\e_u^{n+1}\|_{L^2}^2 								\nonumber	\\
	&\leq C\big(h^{2r+2} + \tau^4 + \|\widetilde{\e}_u^{n+1}\|_{L^2}^2 + \|\widetilde{e}_\rho^{n+1}\|_{L^2}^2 
		+ \|\e_u^{n+1}\|_{L^2}^2\big) + \|\nabla e_\phi^{n+1}\|_{L^2}^2.
	\label{I4-final}		
\end{align}
\vspace{0.0in}

{\bf (6) Estimate for $e_\rho^{n+1}$:}  By letting $\chi_h=e_\rho^{n+1}$ in \eqref{error2}, we arrive at 
\begin{align}
		&\frac1{4\tau}\Big(\|e_\rho^{n+1}\|_{L^2}^2 - \|e_\rho^n\|_{L^2}^2 
			+ \|2e_\rho^{n+1}-e_\rho^n\|_{L^2}^2 - \|2e_\rho^n-e_\rho^{n-1}\|_{L^2}^2	
			+ \|e_\rho^{n+1} - 2e_\rho^n+e_\rho^{n-1}\|_{L^2}^2\Big)		\nonumber	\\
		&	\hspace{0.2in}
			+D\|\nabla e_{\rho}^{n+1}\|_{L^2}^2 + \frac{\sigma}{\epsilon}\|e_\rho^{n+1}\|_{L^2}^2	\nonumber	\\
		&=
			\big(\widetilde{\rho}^{n+1}\u^{n+1} - \widetilde{\rho}_h^{n+1}\u_h^{n+1}, \nabla e_\rho^{n+1}\big)
			+ \frac{\sigma}{\epsilon}\big(R_h\rho^{n+1}-\rho^{n+1}, e_\rho^{n+1}\big) 
			+ \mathfrak{T}_\rho^{n+1}(e_\rho^{n+1})	\nonumber	\\
		&:=\sum_{i=1}^3I_{5i}.	\label{I5}
\end{align}

For $I_{51}$, we obtain
\begin{align}
	I_{51}=&
		\big((\widetilde{\rho}^{n+1}-R_h\widetilde{\rho}^{n+1})\u^{n+1}, \nabla e_\rho^{n+1}\big)
			+ \big(\widetilde{e}_\rho^{n+1}\u^{n+1}, \nabla e_\rho^{n+1}\big)		\nonumber	\\
		&
		+ \big(\widetilde{\rho}_h^{n+1}(\u^{n+1} - \Q_h\u^{n+1}), \nabla e_\rho^{n+1}\big)
			+ \big(\widetilde{\rho}_h^{n+1}\e_u^{n+1}, \nabla e_\rho^{n+1}\big)		\nonumber	\\
	\leq&
		\|\widetilde{\rho}^{n+1}-R_h\widetilde{\rho}^{n+1}\|_{L^2}\|\u^{n+1}\|_{L^\infty}\|\nabla e_\rho^{n+1}\|_{L^2}
			+ \|\widetilde{e}_\rho^{n+1}\|_{L^2}\|\u^{n+1}\|_{L^\infty}\|\nabla e_\rho^{n+1}\|_{L^2}		\nonumber	\\
		&+ \|\widetilde{\rho}_h^{n+1}\|_{L^6}\|\u^{n+1} - \Q_h\u^{n+1}\|_{L^3}\|\nabla e_\rho^{n+1}\|_{L^2}
			+C(h^{1-\frac d6}\|\nabla\e_u^{n+1}\|_{L^2} + \|\e_u^{n+1}\|_{L^2}) \|\nabla e_\rho^{n+1}\|_{L^2}	\nonumber	\\
	\leq&
		C\big(h^{2r+2} + \|\widetilde{e}_\rho^{n+1}\|_{L^2}^2 + \|\e_u^{n+1}\|_{L^2}^2 + h^{2-\frac d3}\|\nabla\e_u^{n+1}\|_{L^2}^2\big)
			+ \frac D4\|\nabla e_\rho^{n+1}\|_{L^2}^2,		\label{I51}
\end{align}
where the last term of $I_{51}$ is estimated as
\begin{align*}
	&\big(\widetilde{\rho}_h^{n+1}\e_u^{n+1}, \nabla e_\rho^{n+1}\big)			\\
	&=\big(-\widetilde{e}_\rho^{n+1}\e_u^{n+1}, \nabla e_\rho^{n+1}\big)	 
		+ \big((R_h\widetilde{\rho}^{n+1} - \widetilde{\rho}^{n+1})\e_u^{n+1}, \nabla e_\rho^{n+1}\big)	
		+ \big( \widetilde{\rho}^{n+1}\e_u^{n+1}, \nabla e_\rho^{n+1}\big)			\\
	&\leq
		\|\widetilde{e}_\rho^{n+1}\|_{L^3}\|\e_u^{n+1}\|_{L^6}\|\nabla e_\rho^{n+1}\|_{L^2}
		+ \|R_h\widetilde{\rho}^{n+1} - \widetilde{\rho}^{n+1}\|_{L^4}\|\e_u^{n+1}\|_{L^4}\|\nabla e_\rho^{n+1}\|_{L^2}	\\
		& \hspace{0.2in}
		+ \|\widetilde{\rho}^{n+1}\|_{L^\infty}\|\e_u^{n+1}\|_{L^2}\|\nabla e_\rho^{n+1}\|_{L^2}		\\
	&\leq
		Ch^{-\frac d6}\|\widetilde{e}_\rho^{n+1}\|_{L^2}\|\nabla\e_u^{n+1}\|_{L^2}\|\nabla e_\rho^{n+1}\|_{L^2}
		+ Ch^r \, h^{-\frac d4} \|\e_u^{n+1}\|_{L^2}\|\nabla e_\rho^{n+1}\|_{L^2}	\\
	&\hspace{0.2in}	
		+ C \|\e_u^{n+1}\|_{L^2}\|\nabla e_\rho^{n+1}\|_{L^2}	\\
	&\leq
		C\big(h^{1-\frac d6}\|\nabla\e_u^{n+1}\|_{L^2} + \|\e_u^{n+1}\|_{L^2}\big) \|\nabla e_\rho^{n+1}\|_{L^2}.
\end{align*}

It is easy to calculate $I_{52}$ and  $I_{53}$ as
\begin{align}
	I_{52} \leq& 
		\frac{\sigma}{\epsilon}\|R_h\rho^{n+1}-\rho^{n+1}\|_{L^2}\|e_\rho^{n+1}\|_{L^2}
			\leq Ch^{2r+2} + \frac{\sigma}{4\epsilon}\|e_{\rho}^{n+1}\|_{L^2}^2,		\label{I52}	\\
	I_{53} =&
		\big(D_\tau(R_h\rho^{n+1}-\rho^{n+1}), e_\rho^{n+1}\big) + \big(D_\tau\rho^{n+1}-\partial_t\rho^{n+1}, e_\rho^{n+1}\big)
			+ \big((\widetilde{\rho}^{n+1}-\rho^{n+1})\u^{n+1}, \nabla e_\rho^{n+1}\big)	\nonumber 	\\
		\leq&
			Ch^{r+1}\|e_\rho^{n+1}\|_{L^2} + C\tau^2\|e_\rho^{n+1}\|_{L^2} + Ch^{r+1}\|\nabla e_\rho^{n+1}\|_{L^2}	\nonumber	\\
		\leq&
			C\big(h^{2r+2} + \tau^4\big) + \frac{D}{4}\|\nabla e_\rho^{n+1}\|_{L^2}^2 
				+ \frac{\sigma}{4\epsilon}\|e_\rho^{n+1}\|_{L^2}^2 \label{I53}
\end{align}

Thus, through \eqref{I51}-\eqref{I53}, equation \eqref{I5} is equal to
\begin{align}\label{I5-final}
	&\frac1{4\tau}\Big(\|e_\rho^{n+1}\|_{L^2}^2 - \|e_\rho^n\|_{L^2}^2 
	+ \|2e_\rho^{n+1}-e_\rho^n\|_{L^2}^2 - \|2e_\rho^n-e_\rho^{n-1}\|_{L^2}^2	\Big)			\nonumber		\\
	&\hspace{0.2in}
		+ \frac D2\|\nabla e_{\rho}^{n+1}\|_{L^2}^2 + \frac{\sigma}{2\epsilon}\|e_\rho^{n+1}\|_{L^2}^2	\nonumber	\\
	&\leq
		C\big(h^{2r+2} + \tau^4 + \|\widetilde{e}_{\rho}^{n+1}\|^2 + \|\e_u^{n+1}\|_{L^2} 
			+ h^{2-\frac d3}\|\nabla\e_u^{n+1}\|_{L^2}^2\big).
\end{align}
\vspace{0.0in}

{\bf (7) Estimate for $e_\phi^{n+1}$:} In order to apply the discrete Grownwall inequality, it remains to handle the estimate of $\|\nabla e_\phi^{n+1}\|_{L^2}$ in \eqref{I4-final}. We adopt $\varphi_h = e_\phi^{n+1}$ in \eqref{error1} and then have
\begin{align*}
	\epsilon\|\nabla e_\phi^{n+1}\|_{L^2}^2 
	=& 
		\big(e_\rho^{n+1}, e_\phi^{n+1}\big) + \big(\rho^{n+1}-R_h\rho^{n+1}, e_\phi^{n+1}\big)	\\
	\leq&
		\|e_\rho^{n+1}\|_{L^2}\|e_\phi^{n+1}\|_{L^2} + \|\rho^{n+1}-R_h\rho^{n+1}\|_{L^2}\|e_\phi^{n+1}\|_{L^2}	\\
	\leq&
		C\big(\|e_\rho^{n+1}\|_{L^2} + h^{r+1}\big)\|\nabla e_\phi^{n+1}\|_{L^2},
\end{align*}
which implies 
\begin{align}\label{I6}
	\|\nabla e_\phi^{n+1}\|_{L^2}  \leq C\big(\|e_\rho^{n+1}\|_{L^2} + h^{r+1}\big).
\end{align}
\vspace{0.0in}

{\bf (8) Final estimate at $t=t^n (n\geq 2)$:} We combine \eqref{I4-final}, \eqref{I5-final} and \eqref{I6} to obtain
\begin{align*}
	&\frac1{4\tau}\Big(\|\e_u^{n+1}\|_{L^2}^2 - \|\e_u^n\|_{L^2}^2 	
		+ \|2\e_u^{n+1}-\e_u^n\|_{L^2}^2 - \|2\e_u^n-\e_u^{n-1}\|_{L^2}^2		\\
	&\hspace{0.2in}			
		+ \|e_\rho^{n+1}\|_{L^2}^2 - \|e_\rho^n\|_{L^2}^2 			
		+ \|2e_\rho^{n+1}-e_\rho^n\|_{L^2}^2 - \|2e_\rho^n-e_\rho^{n-1}\|_{L^2}^2	\Big)					\\
	&\hspace{0.2in}
		+ \frac\eta2\|\nabla\e_u^{n+1}\|_{L^2}^2 
		+ \frac D2\|\nabla e_{\rho}^{n+1}\|_{L^2}^2 + \frac{\sigma}{2\epsilon}\|e_\rho^{n+1}\|_{L^2}^2									\\
	\leq& 
		C\big(h^{2r+2} + \tau^4 + \|\widetilde{\e}_u^{n+1}\|_{L^2}^2 + \|\widetilde{e}_\rho^{n+1}\|_{L^2}^2 
		+ \|\e_u^{n+1}\|_{L^2}^2\big) + Ch^{2-\frac d3}\|\nabla\e_u^{n+1}\|_{L^2}^2 	\\
	&\hspace{0.2in}
		+ C\big(\|e_\rho^{n+1}\|_{L^2}^2 + h^{2r+2}\big).
\end{align*} 
By defining 
\begin{align*}
	{\theta}^{n+1}&= \|\e_u^{n+1}\|_{L^2}^2 +  \|2\e_u^{n+1}-\e_u^n\|_{L^2}^2 
		+ \|e_\rho^{n+1}\|_{L^2}^2 + \|2e_\rho^{n+1}-e_\rho^n\|_{L^2}^2,	\\
	{\xi}^{n+1}& = 	\eta\|\nabla\e_u^{n+1}\|_{L^2}^2 
	+ D\|\nabla e_{\rho}^{n+1}\|_{L^2}^2 + \frac{\sigma}{\epsilon}\|e_\rho^{n+1}\|_{L^2}^2,
\end{align*}
and for some small $h$ satisfying $Ch^{2-\frac d3}\leq \frac \eta4$, we then have
\begin{align*}
	\frac{1}{4\tau}\big(\theta^{n+1} - \theta^n\big) + \frac14\xi^{n+1}	
	\leq& 
		C\big(h^{2r+2} + \tau^4 + \|\widetilde{\e}_u^{n+1}\|_{L^2}^2 + \|\widetilde{e}_\rho^{n+1}\|_{L^2}^2 
		+ \|\e_u^{n+1}\|_{L^2}^2 + \|e_\rho^{n+1}\|_{L^2}^2\big)	\\
	\leq&
		C\big(h^{2r+2} + \tau^4 +  \theta^n +  \theta^{n+1}\big). 
\end{align*}
Hence, an application of the discrete Grownwall inequality leads to the desired result
\begin{align*}
	\theta^{n+1} + \tau\sum_{i=1}^{n+1}\xi^{i} \leq C\big(h^{2r+2} + \tau^4 \big),
\end{align*}
i.e., 
\begin{align}
	\|\e_u^{n+1}\|_{L^2}^2  + \|e_\rho^{n+1}\|_{L^2}^2 \leq C(h^{2r+2} + \tau^4 ),
\end{align}
which recovers the induction assumption \eqref{induction-ass}, for $n=1,2,..., N_T-1$.

Next, we investigate the $L^2$ error estimate of $e_\phi^{n+1}$ by using the duality technique. 
From \eqref{I6} we have
\begin{align}\label{error-gradphi}
	\|\nabla e_\phi^{n+1}\|_{L^2} \leq C\big(h^{r+1} + \tau^2\big).
\end{align}
Let $w\in H^2(\Omega)$  be the unique solution of the equation 
\begin{equation}\label{duality-pde}
		-\Delta w = e_\phi^{n+1}	\qquad \mbox{in}	~ \Omega,	
\end{equation} 
with $ \nabla w \cdot\n= 0$ on boundary $\partial\Omega$ and $(w, 1) = 0$.
From the classic elliptic theory \cite{Brenner2008}, if $e_\phi^{n+1}\in L^2(\Omega)$, then $w$ satisfies 
\begin{align}\label{est_w_H2}
	|w|_{H^2}\leq C\|e_\phi^{n+1}\|_{L^2}.
\end{align} 
Meanwhile, the solution of \eqref{duality-pde} also solves the following variational formulation: Find $w\in H^1(\Omega)$ such that it holds 
\begin{align}\label{duality-var}
	\big(\nabla w, \nabla v\big) = \big(e_\phi^{n+1}, v\big),		\qquad \forall ~ v\in H^1(\Omega),
\end{align}
for which the unique solvability can be guaranteed with the help of the Lax--Miligram theorem.
We take $v=w$ in \eqref{duality-var} to obtain 
\begin{align*}
	\|\nabla w\|_{L^2}^2 = \big(e_\phi^{n+1}, w\big) \leq \|e_\phi^{n+1}\|_{L^2}\|w\|_{L^2}
		 \leq C\|e_\phi^{n+1}\|_{L^2}\|\nabla  w\|_{L^2},
\end{align*}
which together with $\|w\|_{L^2} \leq C\|\nabla w\|_{L^2} \leq C\|e_\phi^{n+1}\|_{L^2}$ yields $\|w\|_{H^2} \leq C\|e_\phi^{n+1}\|_{L^2}$.
Meanwhile, taking $v=e_\phi^{n+1}$ in \eqref{duality-var} and $\varphi_h=R_hw$ in \eqref{error1} yield
\begin{align*}
	\|e_\phi^{n+1}\|_{L^2}^2 
		=& \big(\nabla w, \nabla e_\phi^{n+1}\big)	\\
		=& \big(\nabla w - \nabla R_hw, \nabla e_\phi^{n+1}\big) + {\epsilon}^{-1}\big(e_\rho^{n+1}, R_hw\big)
			+ {\epsilon}^{-1}\big(\rho^{n+1}-R_h\rho^{n+1}, R_hw\big) 	\\
		\leq&
			\|\nabla w - \nabla R_hw\|_{L^2}\|\nabla e_\phi^{n+1}\|_{L^2}
			+  C\|e_\rho^{n+1}\|_{L^2}\|R_hw\|_{L^2} + C\|\rho^{n+1}-R_h\rho^{n+1}\|_{L^2}\|R_hw\|_{L^2}	\\
		\leq&
			Ch\|w\|_{H^2}\big(h^{r+1} + \tau^2\big) + C\big(h^{r+1} + \tau^2\big)\|w\|_{L^2} + Ch^{r+1}\|w\|_{L^2}		\qquad\mbox{(by using \eqref{error-gradphi})}\\
		\leq&
			C\big(1+h\big)\big(h^{r+1}+\tau^2\big) \|w\|_{H^2}	\\
		\leq& 
			C\big(h^{r+1} + \tau^2\big) \|e_\phi^{n+1}\|_{L^2},
\end{align*}
which in turn leads to
\begin{align*}
	\|e_\phi^{n+1}\|_{L^2} \leq C\big(h^{r+1} + \tau^2\big).
\end{align*}

Finally, we obtain the error estimate results
\begin{align*}
	\|\e_u^{n+1}\|_{L^2} + \|e_\rho^{n+1}\|_{L^2} + \|e_\phi^{n+1}\|_{L^2} \leq C(h^{r+1} + \tau^2)
\end{align*}
which together with the triangle inequality can complete the proof of Theorem \ref{thm-error}.

\section{Numerical Examples}\label{sec-numerical}

In this section, we will present some numerical examples to illustrate the theoretical results in Theorem \ref{thm-error} and Theorem \ref{thm-energy}. 

\subsection{Convergence test}

We first consider that, in a unit domain $\Omega=(0, 2\pi)\times(0, 2\pi)$ and the terminal time $T=1$, system \eqref{PDE1}-\eqref{PDE4} admits the unique solution
\begin{equation*}
	\begin{aligned}
		&\phi = t^4\cos(x)\cos(y),	\\
		&\rho = t^4\cos(x)\cos(y),	\\
		&\u = t^4
		\left(
		\begin{aligned}
			\sin^2(x)\sin(2y)	\\
			-\sin(2x)\sin^2(y)	
		\end{aligned}
		\right),		\\
		&p = t^4\cos(x)\cos(y),
	\end{aligned}
\end{equation*}
and the parameters are set as $\varepsilon=D=\sigma=\eta=1$.
Then, to investigate the temporal convergence orders, we first choose $\tau= \frac{1}{10}, \frac{1}{20}, \frac{1}{40}, \frac{1}{80}$ with a sufficiently small $h=\frac{2\pi}{160}$, and obtain the numerical results in Table \ref{tab_error_time}, which are of second-order accuracy as shown in Theorem \ref{thm-error}.  

\begin{table}[h]
	\centering
	\caption{Temporal convergence rates of scheme \eqref{scheme1}-\eqref{scheme4} with $h=2\pi/160$} 
	\begin{tabular}{ccccccccc}
		\toprule
		$\tau$		&$e_\phi$			& rate  & $e_\rho$ 			&rate	&$\e_u$		&rate	\\
		\midrule
		1/10	&2.596e-02   &-		&5.274e-02   &-		&4.404e-02      & -		   		\\
		1/20   &6.879e-03   &1.92  &1.399e-02   &1.91  &1.157e-02      &1.93		\\
		1/40   &1.768e-03   &1.96  &3.600e-03   &1.96  &2.963e-03    &1.97		\\
		1/80   &4.482e-04   &1.98  &9.130e-04	&1.98  &7.496e-04    &1.98 		\\
		\bottomrule
	\end{tabular}
	\label{tab_error_time}
\end{table}

\begin{comment}
\begin{table}[h]
	\centering
	\caption{Temporal convergence rates of scheme \eqref{scheme1}-\eqref{scheme4} with $h=2\pi/160$} 
	\begin{tabular}{ccccccccc}
		\toprule
		$\tau$		&$e_\phi$			& rate  & $e_\rho$ 			&rate	&$\e_u$		&rate	&$e_p$		&rate\\
		\midrule
		1/10	&2.596e-02   &-		&5.274e-02   &-		&4.404e-02      & -	&1.157e-01   	&-   \\
		1/20   &6.879e-03   &1.92  &1.399e-02   &1.91  &1.157e-02      &1.93	&3.224e-02 &1.84 \\
		1/40   &1.768e-03   &1.96  &3.600e-03   &1.96  &2.963e-03    &1.97	&8.532e-03	&1.92 \\
		1/80   &4.482e-04   &1.98  &9.130e-04	&1.98  &7.496e-04    &1.98 	&2.228e-03	&1.94\\
		\bottomrule
	\end{tabular}
	\label{tab_error_time}
\end{table}
\end{comment}

Next, we choose $h=\frac{2\pi}{10}, \frac{2\pi}{20}, \frac{2\pi}{40}, \frac{2\pi}{80}$ by fixing $\tau=\frac{1}{1000}$ to be negligible, and then get the numerical results as shown in Table \ref{tab_error_space}, which are also consistent with the theoretical analysis in Theorem \ref{thm-error}.  

\begin{table}[h]
	\centering
	\caption{Spatial convergence rates of  scheme \eqref{scheme1}-\eqref{scheme4} with $\tau=1/1000$} 
	\begin{tabular}{ccccccccc}
	\toprule
	$h$		&$e_\phi$			& rate  & $e_\rho$ 			&rate	&$\e_u$		&rate	\\
	\midrule
	2$\pi$/10	&1.404e-02   &-		&1.275e-02   &-		&7.523e-02      & -			\\
	2$\pi$/20   &1.760e-03   &3.00  &1.717e-03   &2.89  &1.041e-02      &2.85		\\
	2$\pi$/40   &2.204e-04  &3.00  &2.192e-04  &2.97  &1.352e-03      &2.94		\\
	2$\pi$/80   &2.764e-05  &3.00  &2.816e-05	&2.96  &1.708e-04      &2.98	\\
	\bottomrule
\end{tabular}
	\label{tab_error_space}
\end{table}

\begin{comment}
\begin{table}[h]
	\centering
	\caption{Spatial convergence rates of  scheme \eqref{scheme1}-\eqref{scheme4} with $\tau=1/1000$} 
	\begin{tabular}{ccccccccc}
		\toprule
		$h$		&$e_\phi$			& rate  & $e_\rho$ 			&rate	&$\e_u$		&rate	&$e_p$		&rate\\
		\midrule
		2$\pi$/10	&1.404e-02   &-		&1.275e-02   &-		&7.523e-02      & -	&1.387e-01   	&-   \\
		2$\pi$/20   &1.760e-03   &3.00  &1.717e-03   &2.89  &1.041e-02      &2.85	&2.710e-02 &2.36 \\
		2$\pi$/40   &2.204e-04  &3.00  &2.192e-04  &2.97  &1.352e-03      &2.94		&6.514e-03	&2.06 \\
		2$\pi$/80   &2.764e-05  &3.00  &2.816e-05	&2.96  &1.708e-04      &2.98 	&1.618e-03	&2.01\\
		\bottomrule
	\end{tabular}
	\label{tab_error_space}
\end{table}
\end{comment}

\subsection{Stability and mass conservation}

In this subsection, we still consider a unit domain $\Omega=(0, 2\pi)\times(0, 2\pi)$ and terminal time $T=1$ with the parameters $\varepsilon=D=\sigma=\eta=1$.
We compute the following discrete energy function
\begin{align*}
	\mathcal{E}_h^{n} := \epsilon\|\nabla\phi_h^{n}\|_{L^2}^2 + \epsilon\|2\nabla\phi_h^{n}-\nabla\phi_h^{n-1}\|_{L^2}^2 
		+ \|\u_h^{n}\|_{L^2}^2 + \|2\u_h^{n}-\u_h^{n-1}\|_{L^2}^2,
\end{align*}
as defined in Theorem \ref{thm-energy}, and the charge mass $(\rho_h^n, 1)$ for $n=1,2,...,N_T$.
By defining the initial data as
\begin{equation*}
	\begin{aligned}
		&\rho = t^4\cos(x)\cos(y),	\\
		&\u = t^4
		\left(
		\begin{aligned}
			\sin^2(x)\sin(2y)	\\
			-\sin(2x)\sin^2(y)	
		\end{aligned}
		\right),
\end{aligned}
\end{equation*}
mesh size $h=\frac{2\pi}{50}$, and time step size $\tau=\frac{1}{500}$,
we can obtain the energy decaying and the charge conservation curves in Figure \ref{energy_curve} and  Figure \ref{charge_curve}, respectively, which are consistent with the theoretical analysis in Theorem \ref{thm-energy}.
\begin{figure}[h]
	\centering
	\subfigure[Discrete energy]{\includegraphics[width=0.45\textwidth]{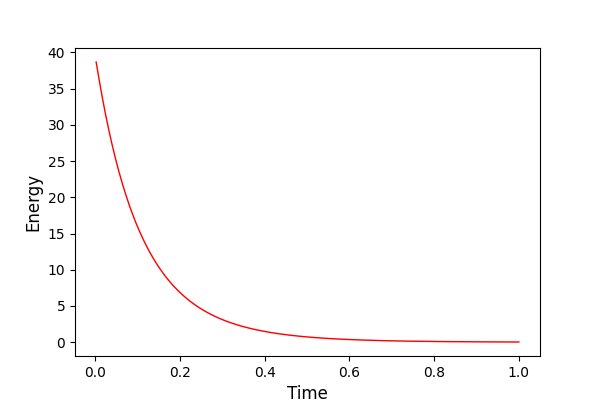}\label{energy_curve}}
	\subfigure[Discrete charge]{\includegraphics[width=0.45\textwidth]{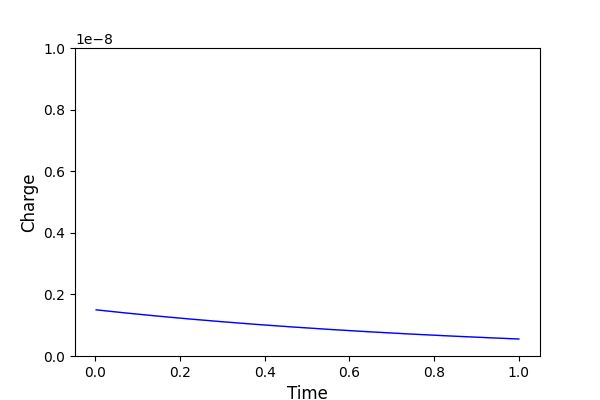}\label{charge_curve}}
	\caption{Discrete energy stability and charge mass conservation at each time level}
	\label{energy_and_charge}
\end{figure}

\begin{comment}
\subsection{The simulation of charge density and fluid velocity}

Finally, we carry out an example to observe the evolution of velocity $u$ and charge density $\rho$ \cite{Li2024}. 
Consider the unit square domain $\Omega=(0, 1)\times(0,1)$ with terminal time $T=2$, and set the parameters $\epsilon=0.25$, $D=0.1$, $\sigma=0.0125$, and $\eta=1$. The initial data is defined as 
\begin{align*}
	\rho^0 = \frac{1}{2}y^2 - y - \cos(2\pi x) + \frac{1}{3},	\qquad
	\u = \0,
\end{align*}
where it is easy to check $(\rho^0, 1)=0$,
Then, with the mesh size $h=\frac{1}{50}$ and time step size $\tau=\frac{1}{100}$, we obtain the distribution of charge density $\rho_h^n$ and velocity $\u_h^n$ in Figure \ref{velocity_and_charge}, where it is observed that the vortexes of the induced velocity is symmetric and the charge density reaches steady state
as time goes on.

\begin{figure}[h]
	\centering
	\subfigure[$t=0.02$]{\includegraphics[width=0.32\textwidth]{fig/rhoh/002.png}\label{charge_distribution1}}
	\subfigure[$t=1$]{\includegraphics[width=0.32\textwidth]{fig/rhoh/100.png}\label{charge_distribution2}}
	\subfigure[$t=2$]{\includegraphics[width=0.32\textwidth]{fig/rhoh/200.png}\label{charge_distribution3}}
	\caption{The distribution of charge density $\rho_h^n$ and velocity $\u_h^n$}
	\label{velocity_and_charge}
\end{figure}
\end{comment}

\section{Conclusion}

In this work, we prove the optimal error estimates of the high-order temporally scheme based on BDF2 and finite element methods for the EHD system.  It is the first work to achieve the optimal convergence rates both in time and space. This analytical frame can also be extended to other decoupled methods and the Cahn--Hilliard--EHD system, by combining the similar techniques developed in \cite{Ma2025, Wang2022} and \cite{Wang2024b}, respectively.

\subsection*{\bf Acknolwedgement}
The work of the authors was partially supported by the National Natural Science Foundation of China (Nos. 12271082, 62231016, 12471371).

\baselineskip 15pt
\renewcommand{\baselinestretch}{0.9}
\footnotesize
%\parindent=6mm
%\section*{References}
%\bibliographystyle{elsarticle-num-names}
\bibliographystyle{plain}
\bibliography{EHDerror}

\end{document}